\documentclass[10pt]{amsart}
\pdfoutput=1 
\allowdisplaybreaks % for proof reading

\usepackage{amsmath,amsfonts,amssymb}
\usepackage{verbatim}
\usepackage{graphicx}
\usepackage{amsthm}

\newcommand\set[1]{\left\{#1\right\}}

\newtheorem{theorem}{Theorem}[section]
\newtheorem{corollary}[theorem]{Corollary}
\newtheorem{proposition}[theorem]{Proposition}
\newtheorem{lemma}[theorem]{Lemma}

\theoremstyle{remark}
\newtheorem{rmk}{Remark}%[section]
\newtheorem{definition}[theorem]{Definition}

\def \bC {\mathbb C}
\def \bD {\mathbb D}

\def \bH {\mathbb H}

\def \bK {\mathbb K}

\def \bN {\mathbb N}
\def \bO {\mathbb O}

\def \bR {\mathbb R}

\def \bR {\mathbb R}
\def \bR {\mathbb R}

\def \bZ {\mathbb Z}
\newcommand{\cp}{\mathbb C}
\newcommand{\rl}{\mathbb R}
\newcommand{\qv}{\mathbb H}

\def \cD {\mathcal D}
\def \cE {\mathcal E}

\def \cG {\mathcal G}
\def \cH {\mathcal H}

\def \cL {\mathcal L}
\def \cM {\mathcal M}

\def \cP {\mathcal P}

\def \cR {\mathcal R}
\def \cS {\mathcal S}
\def \cR {\mathcal R}
\def \cR {\mathcal R}

\def \cV {\mathcal V}

\def \fg {\mathfrak g}

\def \fl {\mathfrak l}

\def \fn {\mathfrak n}

\def \fp {\mathfrak p}

\def \fs {\mathfrak s}

\def \fu {\mathfrak u}
\def \fv {\mathfrak v}

\def \fz {\mathfrak z}
\newcommand{\gt}{\mathfrak}

\def \fS {\mathfrak S}

\def \fU {\mathfrak U}

\newcommand{\SL}{{\rm SL}}
\newcommand{\GL}{{\rm GL}}
\newcommand{\Spin}{{\rm Spin}}
\def \Sp {\text{\rm Sp}}
\def \sp {\mathfrak {sp}}
\def \su {\mathfrak {su}}
\def \so {\mathfrak {so}}
\def \SO {\text{\rm SO}}
\def \SU {\text{\rm SU}}
\def \RE {\text{\rm Re}\,}
\def \IM {\text{\rm Im}\,}
\def \al {\alpha}
\def \la {\lambda}
\def \ph {\varphi}
\def \del {\delta}

\def \lan {\langle}
\def \ran {\rangle}
\def \de {\partial}
\def \trans{\,{}^t\!}
\def \half{\frac12}
\def \inv{^{-1}}

\def \deg {\text{\rm deg\,}}
\def \dim {\text{\rm dim\,}}

\def \tr {\text{\rm tr\,}}
\newcommand{\diag}{{\rm diag}\,}
\newcommand{\Lie}{{\rm Lie\,}}

\begin{document}

\title[]
{Nilpotent Gelfand pairs\\ and spherical transforms of Schwartz functions\\ I. Rank-one actions on the centre}

\author[]
{V\'eronique Fischer, Fulvio Ricci, Oksana Yakimova}

\address {Scuola Normale Superiore\\ Piazza dei Cavalieri 7\\ 56126 Pisa\\ Italy}
\email{v.fischer@sns.it, fricci@sns.it}

\address {Centro di Ricerca Matematica E. De Giorgi\\ Piazza dei Cavalieri 3\\ 56126 Pisa\\ Italy}
\email{yakimova@mccme.ru}
\curraddr{Emmy-Noether-Zentrum, Department Mathematik \\
Universit{\"a}t Erlangen-N{\"u}rnberg, Erlangen, Germany}

%\thanks{Centro De Giorgi}

\subjclass[2010]{Primary: 13A50, 43A32; Secondary:  43A85, 43A90}                         

\keywords{Gelfand pairs, Spherical transform, Schwartz functions, Invariants}

\begin{abstract}
The spectrum of a Gelfand pair of the form $(K\ltimes N,K)$, where $N$ is a nilpotent group, can
be embedded in a Euclidean space $\bR^d$. The identification of the spherical transforms of $K$-invariant Schwartz functions on $N$ with the restrictions to the spectrum of Schwartz functions on $\bR^d$ has been proved already when $N$ is a Heisenberg group and  in the case where $N=N_{3,2}$ is the
free two-step nilpotent Lie group with three generators, with $K={\rm SO}_3$ \cite{ADR1, ADR2, FiRi}. 

We prove that the same identification holds for all pairs in which the $K$-orbits in the centre of $N$ are spheres. In the appendix, we produce bases of $K$-invariant polynomials on the Lie algebra $\fn$ of $N$ for all Gelfand pairs $(K\ltimes N,K)$ in Vinberg's list \cite{V2, Y2}.
\end{abstract}

\maketitle

\makeatletter
\renewcommand\l@subsection{\@tocline{2}{0pt}{3pc}{5pc}{}}
\makeatother

\tableofcontents

\section{Introduction}

\bigskip

Let $N$ be a connected, simply connected, nilpotent Lie group and $K$ a compact group of automorphisms of $N$. 
We are interested in the situation where $(K\ltimes N,K)$ is a Gelfand pair, 
i.e., where the convolution algebra $L^1(N)^K$ of $K$-invariant integrable functions on $N$ is commutative. 

To simplify the notation, 
under these circumstances we will say that $(N,K)$ is a {\it nilpotent Gelfand pair}, n.G.p. in short. 
This notion is equivalent to that of ``commutative nilmanifold'' of \cite[Chapter~13]{W}. 

According to Vinberg's reduction theorem \cite{V1}, the general Gelfand pair $(G,K)$, with $G,K$ Lie groups, is built up over two elementary constituents, a reductive G.p. and a nilpotent G.p. 
The classification of nilpotent Gelfand pairs was obtained in \cite{V2,Y1,Y2}, 
see also \cite[Chapters~13,15]{W}. We point out that in every n.G.p. $N$ is at most step-two \cite{BJR90}. 

The Gelfand spectrum $\Sigma(N,K)$ of $L^1(N)^K$ admits natural homeomorphic embeddings in Euclidean spaces \cite{FeRu}, 
each associated with a finite $d$-tuple $\cD=(D_1,\dots,D_d)$ of self-adjoint generators of the algebra $\bD(N)^K$ of left- and 
$K$-invariant differential operators on $N$. 
The embedded image $\Sigma_\cD\subset\bR^d$ of $\Sigma(N,K)$ is obtained by assigning to each bounded spherical function $\ph$ the $d$-tuple $\big(\mu_1(\ph),\dots,\mu_d(\ph)\big)$ of eigenvalues of $\ph$ as an eigenfunction of $D_1,\dots,D_d$.

It has been proved in several cases that the Gelfand transform establishes the isomorphism 
\begin{equation}\label{conjecture}
\cS(N)^K\cong \cS(\Sigma_\cD)\overset{\text {def}}=\cS(\bR^d)/\{f:f_{|_{\Sigma_\cD}}=0\}\ ,
\end{equation}
and that this condition is independent of the generating set $\cD$ as above, cf. \cite{ADR2}, Section~3. More precisely, \eqref{conjecture} consists of two separate statements:
\begin{equation}\label{multiplier}
\begin{aligned}
&\text{\it Given $u\in\cS(\bR^d)$, there is a unique $F\in\cS(N)^K$ such that $u=\cG F$ on $\Sigma_\cD$,}\\
&\text{\it and $F$ depends continuously on $u$.}
\end{aligned}
\end{equation}
\begin{equation}
\begin{aligned}\label{control}
&\text{\it Given $F\in\cS(N)^K$ and $p\in\bN$, there are $q=q(p)\in\bN$ and $u_p\in\cS(\bR^d)$}\\
&\text{\it such that $\cG F=u_p$ on $\Sigma_\cD$ and $\|u_p\|_{(p)}\le C_p\|F\|_{(q)}$.}
\end{aligned}
\end{equation}

Since \eqref{multiplier} holds on any n.G.p., cf. \cite{ADR2, FiRi}, the main problem is to prove \eqref{control}.
This has been proved in \cite{ADR1,ADR2} for any pair $(N,K)$ in which $N$ is the $(2n+1)$-dimensional Heisenberg group $H_n$, 
and in \cite{FiRi} for the case where $N$ is the free two-step nilpotent group with 
%%step-two and 
three generators and $K=\SO_3$ or $O_3$.

Condition \eqref{conjecture} has an additional interpretation from the point of view of functional calculus on the operators $D_j$. 
The set $\Sigma_\cD$ is the joint $L^2$-spectrum of $D_1,\dots, D_d$, 
and, given a bounded Borel measurable function $m$ on $\Sigma_\cD$, 
the inverse Gelfand transform of $m$ is the convolution kernel $K_m$ of the operator $m(D_1,\dots,D_d)\in\cL\big(L^2(N)\big)$. 

Then \eqref{conjecture} is equivalent to saying that $K_m\in\cS(N)^K$ if and only if $m\in\cS(\Sigma_\cD)$. 
This connection with functional calculus is more than just a remark, 
because it gives a quick proof of \eqref{multiplier}, which in turn intervenes in the proof of \eqref{control}.

The problem of characterising the Gelfand transforms of $K$-invariant Schwartz functions on the Heisenberg group has been addressed, and solved, by other authors in the past, cf. \cite{G} and, in greater generality, \cite{BJR98}. We find that the conditions given, though explicit, are quite technical, as they are expressed in terms of the different parameters identifying bounded spherical functions, and do not suggest a unifying idea for more general cases.

In this paper we present a further step in a general attempt to generalise \eqref{conjecture} to other nilpotent Gelfand pairs. 

Let $\fz$ be the centre of the Lie algebra $\fn$ of $N$. Then the action of $K$ on $\fn$ restricts to $\fz$.
We say that $\fz$ is {\it rank-one} if a generic  $K$-orbit in $\fz$ is of codimension one.

Since the action of $K$ on $\fn$ also restricts to the derived algebra $[\fn,\fn]\subset\fz$, our assumption implies that $[\fn,\fn]=\fz$, i.e., $\fn$ does not contain abelian factors. Let  $\fv$ denote a $K$-invariant complement of $[\fn,\fn]$ in $\fn$. Then $[\fv,\fv]=\fz$. We will also assume that $K$ acts irreducibly on $\fv$.

Our main result is the following.

\begin{theorem}\label{main}
Assume that $\fz$ is rank-one and $\fv$ is irreducible. Then \eqref{multiplier}, \eqref{control}, and hence \eqref{conjecture} hold.
\end{theorem}

We disregard the cases where $N$ is abelian, already discussed in \cite{ADR2}, and, for the same reason, the cases where $\fz$ is one-dimensional (i.e., $N$ is the Heisenberg group).

The list of all n.G.p. with $\fz$ rank-one and $\fv$ irreducible can be derived directly from Vinberg's list, cf. \cite{V2,Y2}. In order to present it in a simplified form, we first state a preliminary ``$K$-reduction theorem'', which extends Section 8 of \cite{ADR2} and which will be proved in Section \ref{reduction}.

\begin{theorem}\label{normal}
Assume that $(N,K_1)$ and $(N,K_2)$ are both nilpotent Gelfand pairs, and that $K_1\triangleleft K_2$. 
If \eqref{conjecture} holds for $(N,K_1)$, it also holds for $(N,K_2)$.
\end{theorem}

We can then limit our attention to the n.G.p. with $\fz$ rank-one, $\fv$ irreducible and $K$ minimal, in the sense that $(N,K_0)$ is not a n.G.p. for any normal subgroup $K_0$ of $K$. The list of such pairs is in Table  \ref{list}. 

Like in the previous proofs of \eqref{conjecture}, a crucial step is the proof of a ``Geller-type formula'', which provides the {\it a-priori} Taylor expansion for the desired Schwartz extension of $\cG F$ on the singular part of $\Sigma_\cD$ (as defined in Subsection \ref{generalities}). 

With respect to the Heisenberg group case, a serious technical complication occurs for the pairs listed in the second and third block of Table \ref{list} (in fact a simpler proof, based on a generalisation of Theorem \ref{normal}, is possible for the pairs in the first block). This complication is due to the presence, in any set of generators of the algebra of $K$-invariant polynomials on $\fn$, of ``mixed polynomials'', i.e. polynomials depending on both the $\fv$- and the $\fz$-variables. By symmetrisation, this has the effect that ``mixed differential operators'' are present in any choice of $\cD$.

Mixed invariants were already present in the case treated in \cite{FiRi}, i.e., line 10 of Table \ref{list}. The method used there, based on an ``exact Geller-type formula'' (Proposition 5.2 and Corollary 5.3 in \cite{FiRi}) and a radialisation argument, turns out to be not so easily extendable to the other pairs (even though the validity of such a formula is implied {\it a posteriori} by Theorem \ref{main}). 

We prove instead a weaker ``Geller-type formula with remainder term'' (Proposition \ref{derivatives})  which is easier to prove and good enough for our purposes. The radialisation argument is replaced by the extension process, modeled on \cite{ADR1, ADR2} and presented in section \ref{proof}.

\bigskip

As we have mentioned already, 
our approach requires minimal information on the spherical functions involved in each case and on the parameters they depend on. 
Reductions to  quotient groups that are Heisenberg or $H$-type and Radon transform arguments are sufficient to provide enough information when needed. 

However, we require full knowledge of which pairs are involved and how Hilbert bases (i.e., systems of generators for the algebra of $K$-invariant polynomials on $\fn$) look like, cf. Table \ref{invariants}.

The proofs required to obtain the Hilbert bases in Table \ref{invariants} are given in the Appendix. For reasons of completeness, and in view of future extensions, Hilbert bases are given not only for the pairs  considered in this paper, but for all pairs in Vinberg's list (i.e. with $\fv$ irreducible, but no restrictions on $[\fn,\fn]$).
Surprisingly, in all  cases the algebra $\rl[\gt n]^K$ of polynomial 
$K$-invariants on $\gt n$ admits an algebraically independent system of generators.

\vskip1cm

\section{Preliminary facts}\label{preliminary}

\bigskip

In this section we give a quick introduction to nilpotent Gelfand pairs and to the tools used in the sequel. We refer to \cite{W} for uncommented statements concerning the general theory of Gelfand pairs.

\bigskip

\subsection{Spherical functions and Gelfand transform}\label{subs-spherical}\hfill

\bigskip

If $K$ is a compact subgroup of ${\rm Aut}(N)$, the space $L^1(N)^K$ of $K$-invariant integrable functions on $N$ is an algebra under convolution.

Assume here that $L^1(N)^K$ is commutative.
This implies that the nilpotent Lie group $N$ is of step at most two
\cite{BJR90}.
The $K$-{\it spherical functions} on $N$ are the smooth functions $\ph$ on $N$ such that
\begin{enumerate}
\item[(i)] $\ph(0)=1$;
\item[(ii)] $\ph$ is an eigenfunction of all operators in $\bD(N)^K$.
\end{enumerate}

For $\ph$ spherical and $D\in\bD(N)^K$, we call $\mu_D(\ph)$ the eigenvalue of $\ph$ when acted upon by $D$.

The multiplicative linear functionals on $L^1(N)^K$ are given by integration against the bounded spherical functions. So the Gelfand spectrum $\Sigma(N,K)$ of $L^1(N)^K$ can be identified with the set of bounded spherical functions. Under this identification, the Gelfand topology on $\Sigma(N,K)$ coincides with the compact-open topology on the set of bounded spherical functions.

The {\it Gelfand transform} of a function $F\in L^1(N)^K$ is the function $\cG F\in C_0\big(\Sigma(N,K)\big)$ given by
$$
\cG F(\ph)=\int_N F(x)\ph(x\inv)\,dx\ .
$$

\medskip

As a consequence of the symmetry of $L^1(N)$ \cite{Lu}, the bounded spherical functions are of positive type. An integral formula expressing them in terms of the irreducible unitary representations of $N$ is given in \cite{BJR90}. 

Given an irreducible unitary representation $\pi$ of $N$ and $k\in K$, let $\pi^k$ be the representation given by $\pi^k(n)=\pi(k\inv n)$ for $n\in N$. Denote by $K_\pi$ the subgroup of those elements $k\in K$ such that $\pi^k\sim \pi$. 

The unitary intertwinings operators give a projective unitary representation $\sigma_\pi$ of $K_\pi$ on the representation space $\cH_\pi$ of $\pi$. Then $\cH_\pi$ decomposes into irreducible $K_\pi$-invariant subspaces without multiplicities (this condition, for generic $\pi$, is in fact equivalent to saying that $(N,K)$ is a n.G.p.). Let
$$
\cH_\pi=\sum_{\sigma\in \widehat{K_\pi}} V_{\pi,\sigma}
$$
be this decomposition, with $V_{\pi,\sigma}$ possibly trivial.

\begin{theorem}\label{BJR}
Assume that $V_{\pi,\sigma}$ is nontrivial. 
Given a unit vector $v\in V_{\pi,\sigma}$, set
\begin{equation}\label{spherical}
\ph_{\pi,\sigma}(x)=\int_K\lan\pi(kx)v,v\ran\,dk\ .
\end{equation}

Then
\begin{itemize}
\item[(i)] $\ph_{\pi,\sigma}$  is $K$-spherical, bounded, and does not depend on $v$;
\item[(ii)] if $\pi'\sim\pi^k$ for some $k\in K$, then $\ph_{\pi,\sigma}=\ph_{\pi',\sigma}$ for every $\sigma\in\widehat{K_\pi}$, modulo the natural identification of $\widehat{K_\pi}$ with $\widehat{K_{\pi'}}$;
\item[(iii)] given any bounded $K$-spherical function $\ph$, there exist $\pi\in\widehat N$, unique modulo the action of $K$, and a unique $\sigma\in\widehat{K_\pi}$ with $V_{\pi,\sigma}$ nontrivial, such that $\ph=\ph_{\pi,\sigma}$.
\end{itemize}
\end{theorem}

See \cite[Section~14.5]{W} for a description of this result in terms of spherical representations of the semi-direct product $K\ltimes N$ and \cite{BR08} for their correspondence with coadjoint orbits of $K\ltimes N$.

\medskip

\subsection{Embeddings of $\Sigma(N,K)$ in Euclidean spaces}\hfill
\label{embedding_in_euclidean_spaces}

\bigskip

Let $\fn$ be the Lie algebra of $N$, $\fS(\fn)$ its symmetric algebra, $\fU(\fn)$ the universal enveloping algebra of $\gt n$. The identification of $\fU(\fn)$ with the algebra $\bD(N)$ of left-invariant differential operators on $N$ will be assumed throughout. 

We will use a modified symmetrisation operator
$\la':\fS(\fn)\longrightarrow \fU(\fn)$, related to the standard symmetrisation $\la$ by the formula
\begin{equation}\label{symm}
\la'(x^\al)=i^{-|\al|}\la(x^\al)\ ,
\end{equation}
for every monomial $x^\al\in \fS(\fn)$. This modification has the effect that the differential operator $D=\la'(P)$ is formally self-adjoint, i.e. 
$$
\int_NDf(x)\overline{g(x)}\,dx=\int_Nf(x)\overline{Dg(x)}\,dx\ ,\qquad \forall\,f,g\in C^\infty_c(N)\ ,
$$ 
if and only if $P$ is real-valued. 

Since $\la'$ is $K$-equivariant, it establishes a linear isomorphism between the $K$-invariant subalgebras $\fS(\fn)^K$ and $\bD(N)^K$. Even though both subalgebras are commutative, $\la'$ is not an algebra isomorphism. Nevertheless, if $\cP\subset \fS(\fn)^K$ is a generating set, the same is true for $\la'(\cP)$ in $\bD(N)^K$. Then Hilbert's basis theorem implies the existence of finite generating sets of $\bD(N)^K$.

In the sequel we will need to consider finite {\it ordered sequences}  $(D_1,\dots,D_d)$ of elements of $\bD(N)^K$ such that the $D_j$ form a generating set. For this reason, we call {\it Hilbert basis} a finite sequence $\rho=(\rho_1,\dots,\rho_d)$ of real-valued $K$-invariant polynomials which generate $\fS(\fn)^K$.
Clearly $\rho$ defines a map
$$
\rho:\fn\longrightarrow \bR^d\ ,
$$
called the {\it Hilbert map}.

To any Hilbert basis $\rho=(\rho_1,\dots,\rho_d)$ we associate the $d$-tuple $\cD=\big(D,\dots,D_d\big)$ of generators of $\bD(N)^K$, where $D_j=\la'(\rho_j)$.

We denote by $\fv$
a $K$-invariant complement of $[\fn,\fn]$ in $\fn$.
Hilbert's basis Theorem also implies the existence of a bi-homogeneous Hilbert basis of $\fn$,
i.e., homogeneous in each of the variables $v\in \fv$ and $z\in [\fn,\fn]$. 
Note that if a polynomial $P$ on $\fn$ is bi-homogeneous,
then its symmetrisation $\lambda'(P)$ is a homogeneous differential operator
where the homogeneity of operators on $N$ is referred to the automorphic dilations $(v,z)\mapsto (rv,r^2z)$ of $N$.
Furthermore, if $P$ is of degree $\gamma'$ and $\gamma''$ in $v$ and $z$ respectively,
then $\lambda'(P)$ is of degree $\gamma=\gamma'+2\gamma''$.

Thus we always have the possibility of constructing a Hilbert basis organised as follows: 
we start with a Hilbert basis $\rho_{\fv}$ of $K$-invariant polynomials on $\fv$
containing the square norm $|v|^2$ on $\fv$, 
and a Hilbert basis $\rho_\fz$ of $K$-invariant polynomials on $[\fn,\fn]$;
then we complete it, if necessary, with an additional set $\rho_{\fv,[\fn,\fn]}$ of polynomials depending on variables from both $\fv$ and $[\fn,\fn]$. 
By symmetristion, this leads to a family of homogeneous generators of $\bD(N)^K$ containing the hypoelliptic operator $\lambda'(|v|^2)$.

Let $P\in\fS(\fn)^K$ be a real-valued polynomial, and $D=\la'(P)$. If $\pi\in\widehat N$, $d\pi(D)$ operates on each $K_\pi$-invariant subspace $V_{\pi,\sigma}$ as $\mu_D(\ph_{\pi,\sigma})I$. Since $V_{\pi,\sigma}$ is finite-dimensional and $d\pi(D)$ is symmetric, this implies that $\mu_D(\ph_{\pi,\sigma})$ is real.

The following theorem is contained in \cite{FeRu}.

\begin{theorem}\label{embedding}
Let $\cD=(D_1,\dots,D_d)$ a generating $d$-tuple of formally self-adjoint elements of $\bD(N)^K$, and let
\begin{equation}\label{SigmaD}
\Sigma_\cD=\big\{\big(\mu_{D_1}(\ph),\dots,\mu_{D_d}(\ph)\big):\ph\in\Sigma(N,K)\big\}\subset\bR^d\ .
\end{equation}

Then $\Sigma_\cD$ is closed in $\bR^d$ and the map
\begin{equation}\label{embeddingformula}
\Phi_\cD:\ph\longmapsto \big(\mu_{D_1}(\ph),\dots,\mu_{D_d}(\ph)\big)
\end{equation}
is 1-to-1 and a homeomorphism of $\Sigma(N,K)$ onto $\Sigma_\cD$.
 \end{theorem}
 
 \medskip

\subsection{Schwartz spaces}\hfill

\bigskip

For functions on $\bR^n$, the $p$-th Schwartz norm is
\begin{equation}\label{p-norm}
\|f\|_{(p)}=\sum_{|\al|\le p}\sup_{x\in\bR^n}(1+|x|)^p|\de^\al f(x)|\ ,
\end{equation}
and $\cS(\bR^n)$ is the space of $C^\infty$-functions $f$ for which $\|f\|_{(p)}<\infty$ for all $p\in\bN$, endowed with the Fr\'echet topology induced by the family of all Schwartz norms. Since this topology does not depend on the choice of coordinates, the Schwartz space can be defined intrinsically on any finite-dimensional real vector space.

Using the fact that the exponential map $\exp:\fn\longrightarrow N$ is a diffeomorphism, we define
$$
\cS(N)=\big\{f:f\circ\exp\in\cS(\fn)\big\}\ .
$$

We refer to \cite{FS} for the basic facts concerning $\cS(N)$, in particular about the possibility of defining $\cS(N)$ intrinsically, with partial derivatives replaced by left-invariant vector fields and Euclidean norm replaced by a homogeneous norm in \eqref{p-norm}.

\medskip

If $E$ is a closed subset of $\bR^d$, we call $\cS(E)$ the space of restrictions to $E$ of functions in $\cS(\bR^d)$, with the quotient topology induced from the identification
$$
\cS(E)=\cS(\bR^d)/\{f:f=0\ {\rm on }\ E\}\ .
$$

The following statement is a consequence of Lemma 3.1 in \cite{ADR2}.

\begin{proposition}\label{D-independence}
Let $\cD=(D_1,\dots,D_d)$ and $\cD'=(D'_1,\dots,D'_{d'})$ two generating families of elements of $\bD(N)^K$. Then the map 
$$
f\longmapsto f\circ(\Phi_\cD\circ\Phi_{\cD'}\inv)
$$
is an isomorphism from $\cS(\Sigma_\cD)$ onto $\cS(\Sigma_{\cD'})$.
\end{proposition}

This implies that the validity of \eqref{conjecture} for a given n.G.p. does not depend of the choice of $\cD$.

\bigskip

\subsection{Spectral multipliers}\hfill

\bigskip

Let $\cD=\{D_1,\dots,D_d)$ be a generating system of formally self-adjoint elements of $\bD(N)^K$. Then the $D_j$ are essentially self-adjoint and $\Sigma_\cD$ is their joint $L^2$-spectrum (cf. \cite{FiRi}, Proposition 3.1).

Therefore, any bounded Borel function $m$ on $\Sigma_\cD$ defines a bounded operator $m(D_1,\dots,D_d)$ on $L^2(N)$ which is left-invariant and commutes with the action of $K$. By the Schwartz kernel theorem, there is 
$$
K_m=m(D_1,\dots,D_d)\del_0\in\cS'(N)^K\ ,
$$
such that
$$
m(D_1,\dots,D_d)F=F*K_m\ .
$$

The following proposition combines together Theorem 5.2 in \cite{ADR2} and Proposition 3.3 in \cite{FiRi}.
\begin{proposition}\label{hulanicki}
Let $D_1,\dots,D_k\in\cD$ be homogeneous, with at least one $D_j$ hypoelliptic. 

If $g\in\cS(\bR^k)$, then $g(D_1,\dots,D_k)\del_0\in\cS(N)^K$.
Furthermore the mapping $g\longmapsto g(D_1,\dots,D_k)\del_0$ is continuous from $\cS(\bR^k)$ to $\cS(N)^K$.
\end{proposition}

Notice that this statement proves \eqref {multiplier} for $\cD$ consisting of homogeneous operators, with one of them hypoelliptic. 
We can always choose such $\cD$ (see subsection \ref{embedding_in_euclidean_spaces}). Thus by Proposition \ref{D-independence}, this gives \eqref {multiplier} for every n.G.p.

\bigskip

\subsection{Abelian pairs}\hfill

\bigskip

We briefly discuss the case of a n.G.p. where $N$ is abelian. Without loss of generality, we can assume that $N=\bR^n$ and that $K$ is a closed subgroup of ${\rm O}_n$. Proofs and details on the following statements can be found in \cite{ADR2}.

If $\rho=(\rho_1,\dots,\rho_d)$ is a Hilbert basis, the $d$-tuple $\cD$ consists of the constant coefficient operators $D_j=\rho_j(i\inv\nabla)$, where $\nabla=(\de_{x_1},\dots,\de_{x_n})$. Then 
$$
\Sigma_\cD=\rho(\bR^n)\ .
$$

The Gelfand transform of a function $F\in L^1(\bR^n)^K$ and its Fourier transform $\hat F$ are related by the identity
\begin{equation}\label{G-F}
\hat F=\cG F\circ\rho\ .
\end{equation}

For abelian pairs, Theorem \ref{main} is a consequence of \eqref{G-F} and of the following result, adapted from \cite{S,M}.

\begin{theorem}\label{Schwarz-Mather}
There exists a continuous linear operator $\cE_K:\cS(\bR^n)^K\longrightarrow \cS(\bR^d)$ such that, for every $G\in\cS(\bR^n)^K$,
$$
G=(\cE_K G)\circ\rho\ .
$$
\end{theorem}

\begin{corollary}\label{abelianmain}
The map $F\longmapsto \cE_K\hat F$ is linear and continuous, and its composition with the canonical projection of $\cS(\bR^d)$ onto $\cS(\Sigma_\cD)$ is $\cG$.
\end{corollary}

This shows that \eqref{control} holds with $u_p=\cE_K\hat F$. This is a strong improvement on \eqref{control}, because $u_p$ is independent of $p$ and it depends linearly on $F$.

\bigskip
\subsection{$H$-type groups and $\fv$-radial functions}\hfill
\bigskip

An {\it $H$-type Lie algebra} is a step-two nilpotent Lie algebra $\fn=\fv\oplus\fz$, endowed with an inner product $\lan\ ,\ \ran$ such that 
\begin{enumerate}
\item $\fz=[\fn,\fn]$;
\item $\fv\perp\fz$;
\item for any $v\in\fv$ and $z\in\fz$, the element $J_zv$ defined by the condition
\begin{equation}\label{Jz}
\lan J_zv,v'\ran=\big\lan z,[v,v']\big\ran\ ,\qquad \forall\,v'\in\fv\ ,
\end{equation}
satisfies the equality
\begin{equation}\label{orthogonal}
|J_zv|=|z||v|\ .
\end{equation}
\end{enumerate}

An $H$-type group is a connected, simply connected Lie group whose Lie algebra is $H$-type.

It is proved in \cite{DR} that the space $L^1_\fv(N)$ of $\fv$-radial integrable functions $F$ on any $H$-type group $N$ (i.e. such that $F(v,z)=F(v',z)$ if $|v|=|v'|$) is a commutative Banach algebra under convolution.

The following statement combines results from \cite{DR} and \cite{ADR1} to show that  $L^1_\fv(N)$ admits a spectral analysis that resembles very much the spherical analysis of nilpotent Gelfand pairs.
One must take into account that, in general, $L^1_\fv(N)$ does {\it not} coincide with $L^1(N)^K$ for any $K$.

\begin{theorem}\label{htype}
Let $\{X_1,\dots,X_m\}$, $\{T_1,\dots,T_n\}$ be orthonormal bases of $\fv$ and $\fz$ respectively (regarded as left-invariant vector fields on $N$). Set $L=-\sum_{j=1}^mX_j^2$. Then  
\begin{enumerate}
\item the multiplicative linear functionals on $L^1_\fv(N)$ are given by integration against the bounded ``spherical'' functions, defined as the bounded $\fv$-radial joint eigenfunctions $\ph$ of $\cD=\{L,i\inv T_1,\dots,i\inv T_n\}$ such that $\ph(0)=1$;
\item the Gelfand spectrum of $L^1_\fv(N)$ is homeomorphic to the set $\Sigma_\cD$ of $(n+1)$-tuples of eigenvalues of the bounded spherical functions, relative to the elements of $\cD$;
\item $\Sigma_\cD$ is also the joint $L^2$-spectrum of the elements of $\cD$;
\item the analogue of \eqref{conjecture} holds, i.e., under the Gelfand transform,
$$
\cS_\fv(N)\cong \cS(\Sigma_\cD)\ .
$$
\end{enumerate}
\end{theorem}

\bigskip

\section{Proof of Theorem \ref{normal}}\label{reduction}\hfill
\bigskip

Let $K_1$ be a subgroup of $K_2$, and assume that $(N,K_1)$ is a nilpotent Gelfand pair. Then also $(N,K_2)$ is a n.G.p., for the simple fact that $L^1(N)^{K_2}$ is a subalgebra of $L^1(N)^{K_1}$. On the other hand, averaging over $K_2$ gives a projection of $L^1(N)^{K_1}$ onto $L^1(N)^{K_2}$.

Therefore, every multiplicative linear functional on $L^1(N)^{K_1}$ induces, by restriction, a multiplicative linear functional on $L^1(N)^{K_2}$. The restriction map is surjective, continuous and open from $\Sigma(N,K_1)$ to $\Sigma(N,K_2)$.

If, in addition, $K_1$ is normal in $K_2$, then $K_2$ acts by composition on $L^1(N)^{K_1}$, inducing an action of the quotient group $W=K_2/K_1$ on $\Sigma(N,K_1)$. In this case, the multiplicative functionals on $L^1(N)^{K_1}$ with the same restriction to $L^1(N)^{K_2}$ form a $W$-orbits, and therefore $\Sigma(N,K_2)\cong\Sigma(N,K_1)/W$.

This last identification can be realised on the embedded copies of the spectra as follows. 

Given $D\in\bD(N)^{K_1}$ and $k\in K_2$, let $D^k\in \bD(N)^{K_1}$ be defined by $D^k(f\circ k)=(Df)\circ k$. This gives an action of $W$ on $\bD(N)^{K_1}$ which respects homogeneity.

Let $\bD_m(N)^{K_1}$ be the homogeneous component of degree $m$ of $\bD(N)^{K_1}$, and choose $\bar m$ large enough so that $\cV=\sum_{m\le\bar m}\bD_m(N)^{K_1}$ generates $\bD(N)^{K_1}$. 

We also choose $\cD_1=(D_1,\dots,D_q)$ a basis of $\cV$, which gives an embedding $\Sigma_{\cD_1}$ of $\Sigma(N,K_1)$ in $\bR^q$. 
Then $W$ acts linearly on $\cV$, $\Sigma_{\cD_1}$ is invariant under this action, and the embedding of $\Sigma(N,K_1)$ intertwines the $W$-actions.   

The choice of the basis $\cD_1$ also introduces coordinates $(\xi_1,\dots,\xi_q)$ on $\cV$.
We fix a Hilbert basis $\rho(\xi)=\big(\rho_1(\xi),\dots,\rho_r(\xi)\big)$ of $\cP(\cV)^W$.

\begin{lemma}\label{hilbertmap} 
For $j=1,\dots,r$, let $D'_j=\rho_j(D_1,\dots,D_q)$. Then $\cD_2=(D'_1,\dots,D'_r)$ is a system of generators for $\bD(N)^{K_2}$ and $\Sigma_{\cD_2}=\rho(\Sigma_{\cD_1})$.
\end{lemma}

\begin{proof} 
It is quite clear that each $D'_j$ is $W$-invariant, hence in $\bD(N)^{K_2}$. 
Take now $D\in\bD(N)^{K_2}\subset\bD(N)^{K_1}$. Then there is $P\in\cP(\cV)$ such that $D=P(D_1,\dots,D_q)$. For every $w\in W$, we also have $D=D^w=(P\circ w)(D_1,\dots,D_q)$. If $\bar P=\int_W(P\circ w)\,dw$, we also have that $D=\bar P(D_1,\dots,D_q)$. Since $\bar P\in\cP(\cV)^W$, it equals $Q\circ\rho$, for some $Q\in\cP(\bR^r)$. Then $D=Q(D'_1,\dots,D'_r)$. This shows that $\cD_2$ generates $\bD(N)^{K_2}$.

The equality $\Sigma_{\cD_2}=\rho(\Sigma_{\cD_1})$ is obvious.
\end{proof}

We can give now the proof of Theorem \ref{normal}, following the lines of \cite{ADR2}, Section 8.

\begin{proof} Denote by $\cG_1$ and $\cG_2$ the spherical transforms of $(N,K_1)$ and $(N,K_2)$ respectively.
Any $F\in\cS(N)^{K_2}$ admits both transforms $\cG_2F$ and $\cG_1 F$, and they are related by the identity
\begin{equation}\label{transforms}
\cG_1 F=\cG_2F\circ\rho_{|_{\Sigma_{\cD_1}}}\ .
\end{equation}

By hypothesis, $\cG_1 F\in \cS(\Sigma_{\cD_1})$. Given $r\in\bN$, there are $s=s(r)\in\bN$ and $u_r\in\cS(\bR^q)$ such that 
\begin{equation}\label{extension}
{u_r}_{|_{\Sigma_{\cD_1}}}=\cG_1 F\ ,\qquad\|u_r\|_{(r)}\le C_r\|F\|_{(s)}\ .
\end{equation} 

Since $F$ is $K_2$-invariant, $\cG_1 F$ is $W$-invariant, so that also $u_r\circ w$ satisfies \eqref{extension} for every $w\in W$. Set $\tilde u_r=\int_W u_r\circ w\,dw$, and $v_r=\cE_Wu_r$, where $\cE_W$ is the linear extension operator provided by the Schwarz-Mather theorem (see Theorem \ref{Schwarz-Mather}). By \eqref{transforms}, $v_r$ is a Schwartz extension of $\cG_2F$. 

Fix now $p\in\bN$. By the continuity of $\cE_W$, there is $r$ such that $\|\cE_Wg\|_{(p)}\le C_p\|g\|_{(r)}$ for every $g\in\cS(\bR^q)^W$. Therefore,
$$
\|v_r\|_{(p)}\le C'_p\|F\|_{(s)}\ .
$$

Passing to the quotient, this provides the inclusion $\cG_2:\cS(N)^{K_2}\rightarrow \cS(\Sigma_{\cD_2})$ and its continuity.
The inverse inclusion and its continuity are guaranteed by Proposition \ref{hulanicki}.
\end{proof}

\bigskip

\section{Rank-one nilpotent Gelfand pairs}\hfill

Many rank-one nilpotent Gelfand pairs can be obtained as a result of the 
following procedure. 

\begin{definition} Let $(N,K)$ be a nilpotent Gelfand pair and $Z_0\subset N$ 
a non-trivial $K$-invariant central subgroup of $N$. Then $(N/Z_0,K)$ is again a Gelfand pair,
see e.g. \cite{V1}. The passage from $(N,K)$ to $(N/Z_0,K)$ and the resulting pair 
 $(N/Z_0,K)$  itself are called {\it central reductions} of $(N,K)$.
\end{definition}

In this section 
we assume that the centre $\fz$ of $\fn$ is rank-one, i.e., that a generic $K$-orbit in $\fz$ has codimension one. As pointed out in the Introduction, this implies that 
$[\fn,\fn]=\fz$, and hence
$$
\fn=\fv\oplus\fz\ .
$$

Up to scalars, $\fz$ admits a unique $K$-invariant inner product, which we assume to be fixed once and for all. 

\bigskip

\subsection{Classification}\label{subsec_classification}\hfill

\bigskip

Table \ref{list} below is derived from Vinberg's list \cite{V2, Y1}
(cf. Appendix). It gives the list of all pairs $(N,K)$ where
\begin{enumerate}
\item[(i)] $\fz$ is rank-one and of dimension higher than one;
\item[(ii)] $K$ is minimal, in the sense that $(N,K_1)$ is not a n.G.p. for any normal subgroup $K_1$ of $K$;
\item[(iii)] $\fv$ is irreducible under the action of $K$.
\end{enumerate}

In view of the results of \cite{ADR2} for the case where $N$ is the Heisenberg group, and of Theorem~\ref{normal}, it will be sufficient to consider the pairs in this list. In fact, one can check from Vinberg's list that if $(N,K)$ is n.G.p. satisfying (i) and (iii) and $K$ admits a normal subgroup $K_1$ such that $(N,K_1)$ is a n.G.p., then (i) and (iii) still hold for this pair.

\begin{table}[htdp]
\begin{center}
\begin{tabular}{|r||l|l|l|l|}
\hline
&$K$&$\fv$&$\fz$&cf. Theorem \ref{generators}\\
\hline
1&$\rm U_1\times\Sp_n $&$\bC^{2n}$&$\bC$& (7) after central reduction\\
2&$\Sp_1\times\Sp_n$&$\bH^n$&$\sp_1$& (9) \\
3&$\text{Spin}_7$&$\bR^8$&$\bR^7$ &(2)\\
\hline
4&$\rm U_2$&$\bC^2$&$ \su_2$& (6) after central reduction\\
5&$\Sp_2$&$\bH^2$&$H\!S^2_0\bH^2$& (7) after central reduction\\
6&$\rm U_2\times \SU_2$&$\bC^2\otimes\bC^2$&$ \su_2$& (11) after central reduction\\
7&$\SU_2\times \SU_n$ ($n\ge3$)&$\bC^2\otimes\bC^n$&$ \su_2$ & (11) after central reduction\\
8&$\SO_2 \times( \Sp_1\times\Sp_n)$&$\bR^2\otimes\bH^{n}$&$ \sp_1$& (12) after central reduction\\
9&$\SO_2\times \text{Spin}_7$&$\bR^2\otimes\bR^8$&$\bR^7$&(8)\\
\hline
10&$\SO_3$&$\bR^3$&$\so_3\cong\Lambda^2\bR^3$&(1)\\
11&$\SU_3$&$\bC^3$&$\Lambda^2\bC^3$&(5)\\
12&$\rm G_2$&$\bR^7$&$\bR^7$&(3)\\
\hline
\end{tabular}
\end{center}
\bigskip
\caption{Pairs with $\dim\fz>1$, $K$ minimal, $\fv$ irreducible, $\fz$ rank-one}
\label{list}
\end{table}

The symbol $\bH$ denotes the quaternionic algebra, $\bO$ the Cayley algebra of octonions, $HS^2_0\bH^2$ the space of quaternionic Hermitian $2\times 2$ matrices with trace zero.

The Lie brackets $[\ ,\ ]:\fv\times\fv\longrightarrow \fz$ and the $K$-actions are as follows (unless specified otherwise, vectors must be understood as columns).

\begin{enumerate}
\item[Line 1.] $[v,u]=\trans vJu$, 
where $J=\begin{pmatrix}0&I_n\\-I_n&0\end{pmatrix}$ is the canonical symplectic form on 
$\bC^{2n}$. 
The factor ${\rm Sp}_n$ of $K$ acts linearly on $\fv$ 
and trivially on $\fz$; 
the torus $\rm U_1$ acts by $(v,z)\longmapsto (e^{it}v,e^{2it}z)$.
\medskip
\item[Line 2.] $[v,u]=\IM(v^*u)$. For $k=(k_1,k_2)\in \Sp_1\times\Sp_n$, 
$k\cdot(v,z)=(k_2v k_1^*,k_1zk_1^*)$.
\medskip
\item[Line 3.] Identify $\bR^8$ with $\bO$ and $\bR^7$ 
with the space of imaginary octonions $\IM \bO$, $[v,u]=\half(vu^* - u v^*)$ 
where multiplication is in $\bO$. 
$K$ acts via the spin representation on $\bO$ and as $\SO_7$ on $\IM\bO$.
\medskip
\item[Lines 4,6,7.] Regarding the elements of $\fv$ as $2\times p$ complex matrices, 
$v=\begin{pmatrix}v_1\\v_2\end{pmatrix}$ with $v_1,v_2\in\bC^p$, we have $[v,u]=vu^*-uv^*-\half\tr(vu^*-uv^*)$. 
With $H$ denoting the second factor of $K$, the action of $k=(k_1,k_2)\in ({\rm S)U}_2\times H$ is
$k\cdot (v,z)=(k_1vk_2^*,k_1zk_1^*)$.
\medskip
\item[Line 5.]  $[v,u]=viu^*-uiv^*-\half\tr(viu^*-uiv^*)$. 
For $k\in {\rm Sp}_2$,
$k\cdot (v,z)=(kv,kzk^*)$.
\medskip
\item[Line 8,9.] Regarding $v$ as $2\times n$ quaternionic matrices (resp. octonionic columns), 
$v=\begin{pmatrix}v_1\\v_2\end{pmatrix}$, 
the Lie bracket is $[v,u]=-\half(  v_1u_1^*-u_1v_1^*+v_2u_2^* - u_2v_2^*)$.
The second factor of $K$ is the same group as at line 2 (resp. 3) 
and acts on $\bH^n$ (resp. $\bO$) in the same way, 
while the first factor is $\SO_2$ acting linearly on the $\bR^2$ factor of $\fv$ and trivially on $\fz$.
\medskip
\item[Lines 10,11.]  $[v,u]=v\wedge u$ and $K$ acts by the natural action on both $\fv$ and $\fz$.
\medskip
\item[Line 12.]  Identifying both copies of $\bR^7$ with $\IM \bO$,
$[v,u]=\half (vu^*-uv^*)$, where multiplication is in $\bO$ and with $\rm G_2$ acting by automorphisms of $\bO$ on both $\fv$ and $\fz$.
\end{enumerate}

\medskip

The table is split into three blocks:
\begin{enumerate}
\item[(i)] the first block contains the pairs in which the action of $K$ on $\fn$ is ``doubly transitive'', cf. \cite{Ko}, i.e. the orbits are products of spheres in $\fv$ and $\fz$ respectively;
\item[(ii)] the second block contains the remaining pairs in which $N$ has square-integrable representations, cf.~\cite{MW};
\item[(iii)] the third block contains the pairs in which $N$ does not have square-integrable representations.
\end{enumerate}
\medskip

We notice that all the algebras $\fn$ in the first two blocks are $H$-type.

\medskip

From Appendix, we choose a Hilbert basis $\rho$ for each pair in Table \ref{list}.
For pairs in the first block, we simply  have
$$
\rho=\big(|v|^2,|z|^2\big)\ .
$$
For the other cases, we have 
$$
\rho=(\rho_\fv,\rho_{\fv,\fz},|z|^2)\ ,
$$
natural choices of $\rho_\fv$ and $\rho_{\fv,\fz}$ being listed in Table \ref{invariants}
with the following notation:
\medskip
\begin{itemize}
\item[Lines 4-7.] The inner product on $\fv$ is $\lan v,u\ran=\RE\tr(vu^*)$, and similarly on $\fz$. 
\medskip
\item[Lines 10,11.]
Let $(e_1,e_2,e_3)$ be the canonical basis of $\fv$ over $\bK=\bR$, resp. $\bC$.
The vectors $[e_2,e_3]$, $[e_3,e_1]$, $[e_1,e_2]$ form a basis of $\fz$ over $\bK$.
Hence we identify 
$\Lambda^2\bR^3$ and $\Lambda^2\bC^3$ with $\bR^3$ and $\bC^3$ respectively.
\medskip
\end{itemize}

\begin{table}[htdp]
\begin{center}
\begin{tabular}{|l||l| l| l|  |l|l| l|  }
\hline
&$\rho_\fv$& $\rho_{\fv,\fz}$\\
\hline
4&$|v|^2$&$iv^*zv$\\
5&$|v|^2$&$v^*zv$\\
6&$|v|^2\,,\,\tr(v^*v)^2$&$ i\tr(vzv^*)$\\
7&$|v|^2\,,\,\tr(v^*v)^2$&$ i\tr(vzv^*)$\\
8&
$|v|^2\, ,\,|v_1|^4+2|v_1v_2^*|^2+|v_2|^4\, ,\,|v_1|^2|v_2|^2- \left(\RE (v_1 v_2^*)\right)^2$&
$\RE \left( (v_1v_2^*-v_2v_1^*)z \right)$
\\
9&$|v|^2\,,\,|v_1|^2|v_2|^2-\left(\RE (v_1 v_2^*)\right)^2$&
$\RE \left( (v_1v_2^*-v_2v_1^*)z \right)$
\\
\hline
10&$|v|^2$&$ \trans v z$\\
11&$|v|^2$&$ \RE(\trans v z) \,,\,\IM (\trans v z)$\\
12&$|v|^2$&$ \trans v z$\\
\hline
\end{tabular}
\end{center}
\bigskip
\caption{}
\label{invariants}
\end{table}

%\begin{table}[htdp]
%\begin{center}
%\begin{tabular}{|l||l| l| l|  |l|l| l|  }
%\hline
%&$K$&$\fv$&$\fz$&$\rho_\fv$& $\rho_{\fv,\fz}$\\
%\hline
%4&$\rm U_2$&$\bC^2$&$ \su_2$&$|v|^2$&$iv^*zv$\\
%5&$\Sp_2$&$\bH^2$&$H\!S^2_0\bH^2$&$|v|^2$&$v^*zv$\\
%6&$\rm U_2\times \SU_2$&$\bC^2\otimes\bC^2$&$ \su_2$&$\tr(v^*v)\,,\,\tr(v^*v)^2$&$ i\tr(vzv^*)$\\
%7&$\SU_2\times \SU_n$ ($n\ge3$)&$\bC^2\otimes\bC^n$&$ \su_2$&$\tr(v^*v)\,,\,\tr(v^*v)^2$&$ i\tr(vzv^*)$\\
%8&$\SO_2\times \Sp_1\times\Sp_n$&$\bR^2\otimes\bH^n$&$ \IM \bH$&
%$|v|^2\, ,\,\tr(v^*v)^2\, ,\,|v_1|^2|v_2|^2-\lan v_1,v_2\ran^2$&$\lan[v_1,v_2]_0,z\ran
%$
%\\
%9&$\SO_2\times \text{Spin}_7$&$\bR^2\otimes\bO$&$\IM\bO$&$|v|^2\,,\,|v_1|^2|v_2|^2-\lan v_1,v_2\ran^2$&$\lan[v_1,v_2]_0,z\ran$\\
%\hline
%10&$\SO_3$&$\bR^3$&$\Lambda^2\bR^3\cong\so_3$&$|v|^2$&$\lan v,z\ran$\\
%11&$\SU_3$&$\bC^3$&$\Lambda^2\bC^3\cong\bC^3$&$|v|^2$
%&$\RE\lan z,v\ran\,,\,\IM\lan z,v\ran$\\
%12&$\rm G_2$&$\IM\bO$&$\IM\bO$&$|v|^2$
%&$\lan z,v\ran$\\
%\hline
%\end{tabular}
%\end{center}
%\bigskip
%\caption{}
%\label{invariants}
%\end{table}

%

\bigskip
\subsection{Generalities about $\Sigma_\cD$}\label{generalities}\hfill

Having chosen $\rho$ as above, we obtain by symmetrisation a generating family $\cD$ of $\bD(N)^K$ consisting~of:
\begin{enumerate}
\item
\label{type1_op} the family $\cL$ of operators corresponding to the elements of $\rho_\fv$; 
$\cL$ always contains the sublaplacian $L$, coming from the symmetrisation of $|v|^2$;
 it also contains one more operator at lines 6,7,9,
and two more operators at line 8;
\item
\label{type2_op} the family $\cM$ of operators corresponding to the elements of $\rho_{\fv,\fz}$;
for the pairs of the first block, $\cM$ is empty;
for the pairs of the second and third block,
it contains a single element $M$, except for the pair at line 11, where $\cM$ consists of two operators $M_1$ and $M_2$.
\item
\label{type3_op}
 the central Laplacian $\Delta$, corresponding to $|z|^2$.
\end{enumerate}

We denote by $d$ the total number of elements of $\cD$. 
Table \ref{invariants} shows that $d$ can be equal to 3, 4 or 5. 
For the elements $\xi\in\bR^d$ we use the double notation $\xi=(\xi_1,\dots,\xi_d)$ and $\xi=(\xi_\cL,\xi_\cM,\xi_\Delta)$. Then $D_d=\Delta$, and we assume that $D_1=L$. 

We split  $\Sigma(N,K)$ as $\Sigma(N,K)^{\text{reg}}\cup\Sigma(N,K)^{\text{sing}}$, 
where the singular part $\Sigma(N,K)^{\text{sing}}$ consists of the bounded spherical functions 
that do not depend on the $\fz$-variable. 
In view of Theorem~\ref{BJR}, 
$\Sigma(N,K)^{\text{sing}}$ contains the bounded spherical functions 
associated to one-dimensional representations of $N$. 
Alternatively, one can say  that $\Sigma(N,K)^{\text{reg}}$ contains the bounded spherical functions associated to coadjoint orbits of maximal dimension. 
We also say that a bounded spherical function is regular, resp. singular, if it belongs to $\Sigma(N,K)^{\rm reg}$, resp. $\Sigma(N,K)^{\rm sing}$.

In our context, transporting the splitting from $\Sigma(N,K)$ to $\Sigma_\cD$, 
we obtain:
\begin{equation}\label{sigmasing}
\Sigma_\cD^{\text{\rm sing}}=\{\xi\in\Sigma_\cD:\xi_\Delta=0\}=\{\xi\in\Sigma_\cD:\xi_\cM=\xi_\Delta=0\}\ .
\end{equation}

Each operator $D_j$, $j=1,\ldots,d$ is homogeneous with degree denoted by $\gamma_j$.
The sublaplacian $L$ is homogeneous of degree $\gamma_1=2$. 
All the other operators of $\cL$, when present, are homogeneous of degree 4. The elements of $\cM$ are homogeneous of degree 4 for the pairs in the second block, and of degree 3 for the pairs in the third block. Finally, $\Delta$ is homogeneous of degree 4.

Then $\Sigma_\cD$ is invariant under the dilations $(\xi_1,\dots,\xi_d)\longmapsto (t^{\gamma_1}\xi_1,\dots,t^{\gamma_d})$ for $t>0$.

Since $L$ is hypoelliptic, the subelliptic estimate $\|D_jf\|_2\le C_j\|L^{\gamma_j/2}f\|_2$, for $f\in\cS(N)$, implies that
the coordinates $\xi_j$, $j=2,\ldots,d$ of the points in $\Sigma_\cD$ satisfy the following inequalities:
\begin{equation}
  \label {control_xi_j}
|\xi_j|\le C_j|\xi_1|^{\frac {\gamma_j}2}.
\end{equation}

On $\Sigma_\cD$ we then have
\begin{equation}\label{norm}
|\xi|\overset{\text {def}}=
\sum_{j=1}^d |\xi_j|^{\frac2{\gamma_j}} \cong\xi_1\ .
\end{equation}

\bigskip

\subsection{Quotient groups}\hfill

\bigskip

Let $\zeta_0\in\fz$ be a unit vector
and let $K'$ be the stabiliser of $\bR\zeta_0$ in $K$. 
Then  $K'=K'_0\rtimes\bZ_2$, where $K'_0$ is the stabiliser of $\zeta_0$ in $K$.
Set also $N'=N/\exp(\zeta_0^\perp)$, where the orthogonal complement is taken in $\fz$. 

By \cite{C}, $(N',K'_0)$ is a n.G.p., that we want to describe case by case.

\subsubsection{The pairs $(N',K'_0)$, first and second block} 
In this case $\fn$ is an $H$-type algebra. More precisely, $\fv$ admits an inner product such that \eqref{Jz} holds. Then the skew-symmetric form $\omega_{\zeta_0}(v,u)=\lan[v,u],\zeta_0\ran=\lan J_{\zeta_0}v,u\ran$ on $\fv$ is nondegenerate, and $N'$ is isomorphic to the Heisenberg group $H_s$, with $2s=\dim\fv$.

Since $J_{\zeta_0}$ is both skew-symmetric and orthogonal, it satisfies the identity $J_{\zeta_0}^2=-I$, i.e. it defines a complex structure on $\fv$.

It follows from \eqref{Jz} that $K'_0$ consists of the elements of $K$ commuting with $J_{\zeta_0}$. 
Therefore, $K'_0$ acts by complex linear transformations on $(\fv,J_{\zeta_0})$, and under the identification $N'\cong H_s$, $K'_0$ is identified with a subgroup of the unitary group ${\rm U}_s$. We keep the notation introduced in Subsection~\ref{subsec_classification}\medskip
\begin{itemize}
\item First block 
\medskip
 \begin{enumerate}
\item[Lines 1,2.] It is easy to verify that $K'_0=\Sp_n$ at line 1, 
and $K'_0={\rm U}_1\times \Sp_n$ at line 2.
\medskip
\item[Line 3.] The stabiliser of $\zeta_0\in\IM\bO$ is isomorphic to $\Spin_6\cong {\rm SU}_4$. Since it acts nontrivially on $(\fv,J_{\zeta_0})$ as a subgroup of ${\rm U}_4$, this can only be ${\rm SU}_4$.
\end{enumerate}
\item Second block: we write $\fv$ as the sum $\fv_1\oplus\fv_2$  
with $\fv_1=\set{\begin{pmatrix}v_1\\0\end{pmatrix}}$,
 $\fv_2=\set{\begin{pmatrix}0\\v_2\end{pmatrix}}$.
With the following choice of $\zeta_0$, for lines 4-6, it is easily checked that 
$\fv_1$ and $\fv_2$ are two equivalent irreducible subspaces of $\fv$  
under the action of $K'_0$
and $J_{\zeta_0}v=\zeta_0v$.
\begin{enumerate}
\item[Lines 4,5.] 
Taking $\zeta_0=\begin{pmatrix}-i&0\\0&i\end{pmatrix}$ for line 4 
and $\zeta_0=\begin{pmatrix}1&0\\0&-1\end{pmatrix}$ for line 5,
then $K'_0$ is $\rm U_1\times \rm U_1$ and $\Sp_1\times \Sp_1$ respectively, 
where the two copies of $\rm U_1$ or $\Sp_1$ 
act independently on the two subspaces $\fv_1$ and $\fv_2$. 
\item[Lines 6,7.] 
Taking $\zeta_0=\begin{pmatrix}-i&0\\0&i\end{pmatrix}$, 
 then $K'_0$ is $({\rm S)( U}_1\times {\rm U}_1)\times \SU_p$
 where $\SU_p$ acts linearly on $\fv_1$, $\fv_2$.
\medskip
\item[Lines 8,9.] We choose the same $\zeta_0$ 
as at line 2 (resp. 3).
The stabiliser of $\zeta_0$ is $\SO_2\times H'$, 
where $H'$ is the stabiliser of $\zeta_0$ for the second factor of $K$; 
hence $H'$ is $\rm U_1\times \Sp_n$ (resp. $\Spin_6$).
The complex structure $J_{\zeta_0}$ allows us to identify 
$\fv_1$ and $\fv_2$ with the same copy of $\bC^{2n}$ (resp. $\bC^4$).
The group $H'$ acts on each of $\fv_1$, $\fv_2$ like $\rm U_1\times \Sp_n$ on $\bC^{2n}$
(resp. $\SU_4$ on $\bC^4$), 
with the elements of $\SO_2\cong{\rm U}_1$ acting as intertwining operators. We denote the Hermitian product on $\bC^{2n}$ (resp. $\bC^4$) as $(\ ,\ )$.
\end{enumerate}
\end{itemize}

\subsubsection{The pairs $(N',K'_0)$, third block} 

In the other cases, the form $\omega_{\zeta_0}$ has a nontrivial radical, which has dimension $r=1$ for the groups at lines 10 and 11, and dimension $r=2$ for the pair at line 11. Then $N'$ is isomorphic to the direct product of $H_{(\dim\fv-r)/2}$ with $\bR^r$.

\medskip
\begin{enumerate}
\item[Lines 10,11.] 
We choose $\zeta_0=e_3$. 
Thus $K'_0$ is isomorphic to $\SO_2$, resp $\SU_2$, 
and, under the action of $K'_0$, $\fv$ splits as the sum of two irreducible subspaces
$\fv_1=\bK e_1\oplus \bK e_2$ and $\fv_2=\bK e_3$.
\medskip
\item[Line 12.] 
Let $(e_1,\ldots,e_7)$ be the canonical basis of $\fv=\IM \bO$.
We choose $\zeta_0=e_7$.
Then $\fv$ splits under $K'_0$ as the sum of $\fv_1=\bR e_1\oplus\ldots\oplus \bR e_6$ and $\fv_2=\bR e_7$.
Left multiplication by $e_7$ gives a complex structure on $\fv_1$,
which can then be identified with $\bC^3$, with $K'_0$ acting on it as $\SU_3$.
\end{enumerate}

\subsubsection{Hilbert bases}

Table \ref{N'}  summarises the result of the two previous subsubsections.
For each pair $(N',K'_0)$, 
it also gives a Hilbert basis $\rho'_\fv$ for the action of $K'_0$ on $\fv$. 
 A Hilbert basis $\rho'_{\fn'}$ for the action of $K'_0$ on $\fn'$ is obtained by adding the coordinate function on $\bR\zeta_0$.

\begin{table}[htdp]
\begin{center}
\begin{tabular}{| r || l | l | l | }
\hline
& $N'$ & $K'_0$  & $\rho'_{\fv}$
\\
\hline
 1&$H_{2n}$ & $\Sp_n$ & $|v|^2$ 
\\
 2&$H_{2n}$ & $\rm U_1\times\Sp_n $ & $|v|^2$ 
\\
 3&$H_4$ & $\SU_4$ & $|v|^2$ 
\\
\hline
 4&$H_2$ & $\rm U_1\times \rm U_1$ & $|v|^2\,,\,|v_1|^2 - |v_2|^2$ 
\\
5&
 $H_4$ & $\Sp_1\times \Sp_1$ & $ |v|^2\,,\,|v_1|^2 - |v_2|^2 $
\\
6&
$H_4$ & ${\rm U_1\times U_1}\times\SU_2$ & $|v|^2\,,\,\tr(v^*v)^2\,,\,|v_1|^2 - |v_2|^2 $
\\
7&
$H_{2n}\, \scriptstyle{(n\ge 3)}$ & ${\rm S(U_1\times U_1})\times\SU_n$ & $|v|^2\,,\,\tr(v^*v)^2\,,\,|v_1|^2 - |v_2|^2 $
\\
8&
$H_{4n}$ & ${\rm U_1}\times ({\rm U_1}\times\Sp_n)$ & $|v|^2\,,\,|v_1|^4\!+\!2|v_1v_2^*|^2\!+\!|v_2|^4
\,,\,|v_1|^2 |v_2|^2 - {(\RE (v_1,v_2))}^2\,,\, -2\IM (v_1,v_2)  $
\\
9&
$H_8$ & $\rm U_1\times \text{SU}_4 $ & $|v|^2\,,\,|v_1|^2 |v_2|^2\!-\! {(\RE (v_1,v_2))}^2\,,\, -2\IM (v_1,v_2)  $
\\
\hline
10&
$H_1\times\bR_{v_3}$ & $ \rm U_1 $ & $ |v|^2\, , \, v_3$ \\
11&
$H_2\times \bC_{v_3} $ & $ \SU_2 $ & $ |v|^2\, , \, \RE v_3 \, , \, \IM v_3$
\\
12&
$H_3\times \bR_{v_7}$ & $ \SU_3 $ & $  |v|^2\, , \, v_7 $\\
\hline
\end{tabular}
\end{center}
\bigskip
\caption{The quotient groups $N'$, with $K'_0$ and $\rho'_\fv$}
\label{N'}
\end{table}

\subsubsection{Differential operators}

In each of the cases under consideration, 
we take the generating system $\cD'_0=(D'_1,\ldots,D'_d)$ of $\bD(N')^{K'_0}$ obtained by symmetrisation of the Hilbert basis $\rho'$ formed by the invariants in $\rho'_\fv$ in Table \ref{N'}, with the addition of the coordinate function on $\bR\zeta_0$. Notice that the cardinality of $\rho'$ is the same cardinality $d$ of the corresponding Hilbert basis $\rho$ on $\fn$.

As for the operators of $\cD$,
each operator $D'_j$, $j=1,\ldots,d$ is homogeneous with degree denoted by $\gamma'_j$
and $D'_1$ is hypoelliptic and homogeneous of degree $\gamma'_1=2$. 
The coordinates $\eta_j$, $j=1,\ldots,d$ of the points in $\Sigma_{\cD_0'}$ satisfy the following inequalities:
\begin{equation}
  \label{control_eta_j}
|\eta_j|\le C'_j|\eta_1|_1^{\frac {\gamma'_j}2}
\end{equation}

\bigskip

\subsection{Radon transforms and structure of $\Sigma_\cD^{\text{\rm reg}}$}\hfill

\bigskip

The notation used throughout assumes that $N$ and $N'$ are identified with their Lie algebras via exponential coordinates. By $\zeta_0^\perp$ we denote the orthogonal of $\zeta_0$ in $\fz$.

Following \cite{FiRi},
for any integrable $K$-invariant function $F$ on $N$, 
we set: 
$$
\cR F(v,t)=\int_{ \zeta_0^\perp} F(v,t\zeta_0+\zeta') d\zeta'
\ , \quad (v,t)\in N'\ ,
$$

This defines a function $\cR F\in L^1(N')^{K'}$, that we call the {\it Radon transform} of $F$ (more precisely, it encodes all the Radon transforms, in the ordinary sense, of the functions $F(v,\cdot)$ on $\fz$).
The operator $\cR$ is linear, $K'$-equivariant and respects convolution.
Furthermore it maps continuously $L^1(N)^K$ to $L^1(N')^{K'}$, and $\cS(N)^K$ to  $\cS(N')^{K'}$.

We extend the definition of $\cR$ to $\bD(N)^K$ in the following way:
if $D\in \bD(N)^K$, then we define $D'=\cR D \in \bD(N')^{K'}$ by 
$$
(D'G)\circ q = D (G\circ q),
\quad G\in C^\infty_c(N'),
$$
where $q:N\rightarrow N'$ is the quotient mapping.
Easy changes of variables show:

\begin{lemma}\label{RadonD}
If $D\in\bD(N)^K$ is the symmetrisation of the polynomial $P$ on $\fn$, then $D'=\cR D$ is the symmetrisation of
the restriction $P_{|_{\fn'}}$ of $P$ to $\fn'$.
\end{lemma}

We recall that, by \cite{C}, $(N'K')$ and $(N',K'_0)$ are n.G.p. We say that a bounded spherical function for $(N',K')$ or $(N',K'_0)$ is regular if it is associated to an infinite-dimensional representation of $N'$. 

By the rank-one action of $K$ on $\fz$, every irreducible unitary representation of $N$ factors through a representation of a conjugate of $N'$ under $K$. Via Theorem \ref{BJR}, this implies that every bounded spherical function for $(N,K)$ can be expressed as
\begin{equation}
  \label{spherical'}
\ph^\sharp(v,z)=\int_K\ph(kv, \lan kz,\zeta_0\ran)\,dk\ ,  
\end{equation}
 and that this is a 1-to-1 correspondence between $\Sigma(N',K')^{\rm reg}$ and $\Sigma(N,K)^{\rm reg}$.

It is convenient to regard $\cR$ as an operator from $\cR:L^1(N)^K$ to $L^1(N')^{K'_0}$, disregarding the extra invariance under $\bZ_2$ at this stage.

A $K'_0$-bounded spherical function on $N'$ has the form
\begin{equation}\label{lambda}
\psi(v,t)=\psi_0(v)e^{i\la t}\ ,
\end{equation}
with $\la\ne0$ if and only if $\psi$ is regular. 
Choose $w\in K'\setminus K'_0$
where $w(v,t)=\big(w_0 (v),-t\big)$ . Then also
$$
\psi^w(v,t)=\psi\circ w(v,t)=(\psi_0\circ w_0)(v)e^{-i\la t}
$$
is a $K'_0$-bounded spherical function and the map
$$
\psi\longmapsto \half(\psi+\psi^w)
$$
is a 2-to-1 correspondence between $\Sigma(N',K'_0)^{\rm reg}$ and $\Sigma(N',K')^{\rm reg}$ (cf. Section \ref{reduction}).

Denote by $\Sigma(N',K'_0)^{\rm reg, +}$ the subset of $\Sigma(N',K'_0)^{\rm reg}$ consisting of the functions \eqref{lambda} with $\la>0$. Then $\Sigma(N',K'_0)^{\rm reg, +}$ is open in $\Sigma(N',K'_0)$ and 
$\Sigma(N',K'_0)^{\rm reg}$ is the disjoint union of $\Sigma(N',K'_0)^{\rm reg, +}$ and $w\Sigma(N',K'_0)^{\rm reg, +}$.

We denote by $\cG'_0$ the Gelfand transform for $(N',K'_0)$. The following statement summarises our previous remarks.

\begin{proposition}\label{spherical_N_N'}
If $\psi$ is a bounded spherical function in $(N',K'_0)$,
then \eqref{spherical'}
defines a bounded spherical function $\psi^\sharp$ for $(N,K)$. For $F\in L^1(N)^K$,
\begin{equation}\label{integrals}
\int_N F(v,z)\psi^\sharp(-v,-z)\,dv\,dz=\int_{N'}\cR F(v,t)\psi(-v,-t)\,dv\,dt\ ,
\end{equation}
i.e., $\cG F(\psi^\sharp)=\cG'_0 (\cR F) (\psi)$. 

For any $D\in \bD(N)^K$, 
$\cR D$ and $D$ share the same eigenvalue on $\psi$ and $\psi^\sharp$ respectively.
 In particular, $\psi^\sharp$ is regular if and only if $\psi$ is regular.  

Moreover, $(\psi^w)^\sharp=\psi^\sharp$ and the map $\psi\longmapsto\psi^\sharp$ is a bijection from $\Sigma(N',K'_0)^{\rm reg, +}$ onto $\Sigma(N,K)^{\rm reg}$.
\end{proposition}

Comparing the restriction to $\fn'$ of the elements of $\rho$ with the elements of $\rho'$, cf. Tables \ref{invariants} and \ref{N'}, we see that every element of $\rho'$ is the restriction to $\fn'$ of an element of $\rho$. If $\cD$, $\cD'$ are the corresponding systems of differential operators, it follows that $\cD'=\cR\cD$.

In the next statement the map $\Theta$ is defined so that $\Theta\inv$ is the same as the map $\ph\mapsto\ph^\sharp$ in \eqref{spherical'} for the embedded copies of the spectra in $\bR^d$. The notation is set so that $\Sigma_\cD\cong\Sigma(N,K)$ is in the $\xi$-space, while $\Sigma_{\cD'_0}\cong\Sigma(N',K'_0)$ is in the $\eta$-space. The coordinates $\xi_1,\eta_1$ represent the eigenvalues of the sublaplacians on $N,N'$ resp., while $\xi_d$, resp. $\eta_d$, represent  the eigenvalues of the central Laplacian on $N$, resp. the derivative in the direction of $\bR\zeta_0$.

\begin{lemma}\label{Sigmas}
Let $\Theta: \bR^{d-1}\times \bR\backslash \{0\} \rightarrow \bR^d$
be given by
$\Theta(\xi_1,\ldots,\xi_d)=(\eta_1,\ldots,\eta_d)$, where  $\eta_1=\xi_1$, $\eta_d=\sqrt\xi_d$ in all cases, and:
\begin{equation}\label{Theta}
\begin{matrix}
\eta_2= \xi_2 \xi_3^{-\frac12} & & &\text{\rm at lines } 4,5,10,11,12&(d=3)\hfill\\
\eta_2=\xi_2\hfill& \eta_3= \xi_3 \xi_4^{-\frac12} & &\text{\rm at lines } 6,7,9\hfill&(d=4)\hfill\\
\eta_2=\xi_2\hfill& \eta_3=\xi_3\hfill&\eta_4= \xi_3 \xi_5^{-\frac12} & \text{\rm at line } 8\hfill&(d=5)\hfill\\
\eta_2=\xi_2 \xi _4^{-\frac 12}& \eta_3= \xi_3 \xi _4^{-\frac 12} && \text{\rm at line } 11\hfill&(d=4)\ .
\end{matrix}
\end{equation}

%$$
%\begin{array}{| c || c | l r |  }
%\hline
%\mbox{line}&d&\mbox{relations}&\\
%\hline
% 4\& 5 \& 10 \& 12 &3& \eta_2= \xi_2 \xi_3^{-\frac12} &\\
%6 \&7 \&8\& 9&4&\eta_2=\xi_2& \eta_3= \xi_3 \xi_4^{-\frac12}\\
%11&4& \eta_2=\xi_2 \xi _4^{-\frac 12} & \eta_3= \xi_3 \xi _4^{-\frac 12}\\
%\hline
%\end{array}
%$$

Then $\Theta$ maps homeomorphically $\Sigma_\cD^{\text{\rm reg}}$ 
onto $\Sigma_{\cD'_0}^{\text{\rm reg},+}$.
\end{lemma}

\vskip1cm

\section{A Geller-type formula with remainder term}
\label{a_geller_type_lemma}

\bigskip

We begin now the proof of Theorem \ref{main}. As we have poined out in the Introduction, it suffices to prove \eqref{control}. We have also commented that we treat all pairs in Table \ref{list} in a unified way. For the pairs in the first block a separate simpler proof is possible, based on Theorem \ref{normal} (but with a variant for the pair at line 3).

Due to the absence of mixed invariants, some of the statements below are either vacuous or considerably simpler to prove if specialised to pairs in the first block.

Proposition \ref{derivatives} below is our key tool. 
It is a (less refined) analogue of Theorem 2.2 in \cite{ADR1}, Proposition 7.2 in \cite{ADR2} and Proposition 5.2 in \cite{FiRi}. 
The structure of the result is derived from \cite{G}. 
It hints that, for $F\in\cS(N)^K$, $\cG F$ has some kind of Taylor development at $\Sigma_\cD^{\text{sing}}$.

\begin{proposition}\label{derivatives}
Let $(N,K)$ be one of the pairs of Table \ref{list}.
Given $F\in\cS(N)^K$, there exist functions $g_{jk}\in\cS(\bR^{d_\fv})$, $j\in\bN^{\#\cM}$, $k\in\bN$, such that,  for every $n\in\bN$,
\begin{equation}\label{geller}
F=\sum_{|j|+2k\le n}\frac1{j!k!}\cM^j\Delta^kg_{jk}(\cL)\del_0+R_n\ ,
\end{equation}
with $R_n\in\cS(N)^K$,
\begin{equation}\label{remainder}
R_n=\sum_{|\al|=n+1}\de_z^\al R'_{\al}\ ,
\end{equation}
with $R'_\al\in\cS(N)$. Moreover, each $g_{jk}$ and $R'_\al$ depends linearly and continuously on $F$.
\end{proposition}

Notice that for the pairs in the first block, $\cM$ is empty, so that \eqref{geller} only contains the sum over $k$.

\bigskip

\subsection{Tools for the proof of Proposition~\ref{derivatives}}\hfill

\bigskip
Let us set first some notation for the pairs in the second and third block of Table \ref{list}.
For each $n\in \bN$, we denote by $E_n$ the set consisting of the following pairs:
\begin{equation}\label{E_n}
  \begin{cases}
(2m+1, 0) , (2m-1,1) , \ldots , (1,m) , (0, m+1) 
& \text{if }n=2m\\
(2m+2,0) , (2m,1) , \ldots , (0,m+1)
& \text{if }\, n=2m+1\ ,
      \end{cases}
 \end{equation}
in all cases but line 11, where we define 
$E_n$ as the set of triples $(j,k)=(j_1,j_2,k)\in \bN^3$ 
such that $(|j|,k)$ runs over the set given in \eqref{E_n}.

\begin{lemma}\label{vanishing}
Let $n\in \bN$ be fixed. With $(\omega,\zeta)$ denoting elements of $\fv\times\fz$, suppose that  $u(\omega,\zeta)\in\cS(\fn)^K$ vanishes at $\zeta=0$ with all derivatives in $\zeta$ up to order $n$.
Then
\begin{equation}\label{expansion}
u(\omega,\zeta)=\sum_{(j,k)\in E_n}\rho_{\fv,\fz}(\omega,\zeta)^j |\zeta|^{2k}h_{jk}\big(\rho(\omega,\zeta)\big)
\end{equation}
where the $h_{jk}\in\cS(\fn)^K$ depend linearly and continuously on $u$.

For a pair in the first block, \eqref{expansion} reduces to
$$
u(\omega,\zeta)=|\zeta|^{2[n/2]+1}h_n(\omega,\zeta)\ ,
$$
with $h_n\in\cS(\fn)^K$ depend linearly and continuously on $u$.
\end{lemma}

The proof of Lemma~\ref{vanishing} is based on the following two simple lemmas, the first of them being an adapted version of Hadamard's lemma.

\begin{lemma}\label{hadamard}
Let  $h\in \cS(\bR_x^n\times\bR_y^m)$.
If $h$ vanishes at each point $(x,0)$ with all derivatives up to order $k$,
then there exist $h_\alpha\in \cS(\bR^n\times\bR^m)$, $\alpha\in \bN^m$, $|\al|=k+1$,
depending linearly and continuously on $h$, such that
$$
h(x,y)=\sum_{\alpha} y^\alpha h_\alpha(x,y)\ .
$$ 
\end{lemma}

\begin{proof}[Sketchy proof] It suffices to consider the case $k=0$.
Fixing  $\psi\in C^\infty_c(\bR^m)$ and $\psi=1$ on a neighborhood of the origin, we have:
$$
\begin{aligned}
h(x,y)
&=\int_0^1 \partial_r h\big(x,ry\big) dr
\\
&=\sum_{j=1}^my_j\left(\psi(y)\int_0^1 (\partial_{y_j} h)\big(x,ry\big) dr
+
y_j h(x,y) \frac{1-\psi(y)}{|y|^2}\right) \ .\qedhere
\end{aligned}
$$
\end{proof}

\begin{lemma}\label{basis_polynomial_sphere}
  \begin{itemize}
  \item[1.] 
Let $a\in \bR^m \backslash \{0\}$, $m\ge2$.
The polynomials $\lan a,\tau \ran^j$, $j\in \bN$, as functions of $\tau$ in the unit sphere $S^{m-1}\subset\fz$, are linearly independent.
  \item[2.] 
Let $a\in \bC^m \backslash \{0\}$, $m\ge2$.
The polynomials $( a\mid \tau)^{j_1}\overline{( a\mid\tau)}{}^{j_2} $, $j\in \bN^2$, 
as functions of $\tau\in S^{2m-1}$, are linearly independent.
  \end{itemize}
\end{lemma}

\begin{proof}[Sketchy proof]  In both cases we can assume that $a=(1,0,\dots,0)$. 
In the real case, if $P(\tau)=\sum_{j=0}^nc_j\lan a,\tau\ran^j$ is zero on the unit sphere, then $\sum_{j=0}^nc_j\tau_1^j$ is zero for $\tau_1\in[-1,1]$.
The same holds in the complex case, with $\tau_1$ in the unit disk in $\bC$.
\end{proof}

\begin{proof}[Proof of Lemma~\ref{vanishing}] We begin with a pair in the second or third block of Table \ref{list}, except for the one at line 11.

In this case there is only one mixed operator $\rho_{\fv,\fz}$ which can be written as:
\begin{equation}\label{p}
\rho_{\fv,\fz}(\omega,\zeta)=p (\omega,\zeta)=\sum_{\ell=1}^m\zeta_\ell p_\ell(\omega)
\end{equation}
where $m=\dim\fz$.

Consider first the case $n=0$. Then $u$ is $K$-invariant and $u(\omega,0)=0$ for every $\omega$. Then, with $h=\cE_Ku$,
$$
u(\omega,\zeta)=h\big(\rho(\omega,\zeta) \big)\ .
$$

It follows that $h(\xi_\cL,0,0)=0$ whenever $\xi_\cL\in\rho_\fv(\fv)$. 
Replacing $h$ by $h(\xi)-h(\xi_\cL,0,0)\psi(\xi_M,\xi_\Delta)$, 
with $\psi\in C^\infty_c(\bR^2)$ and $\psi=1$ on a neighborhood of the origin, 
we may assume that $h(\xi_\cL,0,0)=0$ for every $(\xi_\cL,0,0)\in\bR^d$.

By Lemma \ref{hadamard},
$$
h(\xi)=\xi_M h_1(\xi)+\xi_\Delta h_2(\xi)\ ,
$$
with $h_1,h_2\in\cS(\bR^{d+2})$ depending linearly and continuously on $h$, hence on $u$. Then
\begin{equation}\label{n=0}
u(\omega,\zeta)=p(\omega,\zeta)h_1\big(\rho(\omega,\zeta)\big)+|\zeta|^2h_2\big(\rho(\omega,\zeta)\big)\ .
\end{equation}

This gives \eqref{expansion} for $n=0$.

Suppose now that \eqref{expansion} holds for $n$, 
and suppose that $u$ vanishes at $\zeta=0$ with all derivatives in $\zeta$ up to order $n+1$.

We differentiate \eqref{expansion} $n+1$ times along a given direction $\tau$ in $\fz$ at $\zeta=0$ 
and impose that the result is zero. 
By the $K$-invariance of each term in \eqref{expansion}, 
we can assume that $\tau=e_1$, i.e., we differentiate in $\zeta_1$.  
Each term $p(\omega,\zeta)^j |\zeta|^{2k}$ appearing in \eqref{expansion} is homogeneous of degree $n+1$ in $\zeta$, except for the single term $|\zeta|^{2(m+1)}$ when $n=2m$, which is homogeneous of degree $n+2$. Therefore,
$$
{\de_{\zeta_1}^{n+1}}u(\omega,0)=\!\!\sum_{j+2k=n+1} \!\!{\de_{\zeta_1}^{n+1}}\big(p(\omega,\zeta)^j |\zeta|^{2k}\big)(\omega,0)\,\,h_{jk}\big(\rho(\omega,0)\big)\ .
$$

For $j+2k=n+1$ we have
$$
\de_{\zeta_1}^{n+1}\big(p(\omega,\zeta)^j |\zeta|^{2k}\big)={{n+1}\choose j}j!p_1(\omega)^j(2k)!=(n+1)!p_1(\omega)^j\ .
$$

Therefore, for every $\omega\in\fv$,
$$
\sum_{j+2k=n+1}p_1(\omega)^j h_{jk}\big(\rho_\fv(\omega), 0,0\big)=0\ .
$$

Noticing that $p(\omega,\zeta)=p_1(\omega)$, by $K$-invariance we have
$$
\sum_{j+2k=n+1}p(\omega,\tau)^j h_{jk}\big(\rho_\fv(\omega), 0,0\big)=0\ ,
$$
for every $\tau\in S_\fz$ and every $\omega\in\fv$. 
Fix $\omega\in\fv$ and set $p(\omega,\tau)=\lan a, \tau \ran$,
with  $a=\big(p_1(\omega),\dots,p_m(\omega)\big) \in \bR^m$. Then $a$ is generically non-zero and, by Lemma~\ref{basis_polynomial_sphere}, part 1, 
$h_{jk}\big(\rho_\fv(\omega), 0,0\big)=0$ for any $(j,k)$ with $j+2k=n+1$.

If $n$ is odd, the pairs $(j,k)$ with $j+2k=n+1$ exhaust $E_n$. Then, applying \eqref{n=0},
$$
u(\omega,\zeta)=\sum_{j+2k=n+1}p(\omega,\zeta)^j |\zeta|^{2k}\Big(p(\omega,\zeta)h'_{jk}\big(\rho(\omega,\zeta)\big)+|\zeta|^2h''_{jk}\big(\rho(\omega,\zeta)\big)\Big)\ ,
$$
where each term contains a factor $p(\omega,\zeta)^{j'} |\zeta|^{2k'}$ with $(j',k')\in E_{n+1}$.

For $n$ even the argument is the same, once we have noticed that the pair $(0,m+1)$ belongs to both $E_{2m}$ and $E_{2m+1}$.

\medskip
Consider now the pair at line 11 in Table \ref{list}.
In this case let $q_1$, $q_2$ be the two elements of $\rho_{\fv,\fz}$ in Table \ref{invariants}
and set:
$$
p(\omega,\zeta)=q_1(\omega,\zeta)+iq_2(\omega,\zeta)=\sum_{\ell=1}^3 p_\ell(\omega) \zeta_j\ .
$$

Then \eqref{expansion} is equivalent to:
\begin{equation}\label{expansion_l11}
u(\omega,\zeta)=\sum_{(j,k)\in E_n}p(\omega,\zeta)^{j_1} \overline{ p(\omega,\zeta)}{}^{j_2} |\zeta|^{2k}h_{jk}\big(\rho(\omega,\zeta)\big)
\end{equation}
where the $h_{jk}\in\cS(\fn)^K$ are new functions depending linearly and continuously on $u$.
 
The case $n=0$ is a simple adaptation of the proof for the other cases.

Suppose now that \eqref{expansion} holds for $n$, 
and suppose that $u$ vanishes at $\zeta=0$ with all derivatives in $\zeta$ up to order $n+1$. 
By the inductive assumption, we know that \eqref{expansion} and thus \eqref{expansion_l11} hold. 

In this proof
we denote by $\partial_\tau$ and $\bar\partial_\tau$ the holomorphic and antiholomorphic derivatives 
in a given direction $\tau$ in $\fz\sim\bC^3$. 
We impose the condition $\de_\tau^{r_1}\bar \de_\tau^{r_2}u(\omega,0)=0$, for $r_1+r_2=n+1$.

By $K$-invariance, we can assume that $\tau=e_1$. We have
$$
\de_{\zeta_1}^{r_1}\bar \de_{\zeta_1}^{r_2}\big(p(\omega,\zeta)^j |\zeta|^{2k}\big)= 
\begin{cases}r_1! r_2! p_1(\omega)^{j_1} \overline{p_1(\omega)}{}^{j_2}&\text{ if }j_1=r_1-k\ ,\ j_2=r_2-k\\
0&\text{ otherwise.}
\end{cases}
$$

Arguing as in the previous proof, this implies that for every $r_1,r_2\ge k$ with $r_1+r_2=n+1$,
$$
\sum_{|j|+2k=n+1\,,\,j_1-j_2=r_1-r_2} p_1(\omega)^{j_1} \overline{p_1(\omega)}{}^{j_2} h_{jk}\big(\rho_\fv(\omega), 0,0\big)=0\ ,
$$

By $K$-invariance, for the same values of $r_1,r_2$,
$$
\sum_{|j|+2k=n+1\,,\,j_1-j_2=r_1-r_2} p(\omega,\tau)^{j_1} \overline{p(\omega,\tau)}{}^{j_2} h_{jk}\big(\rho_\fv(\omega), 0,0\big)=0\ ,
$$
for every $\tau\in S^5$ and every $\omega\in\fv$. 
The proof continues as in the previous case, using part 2 of Lemma~\ref{basis_polynomial_sphere}.

\medskip

Finally, the case of a pair in the first block of Table \ref{list} is much simpler to prove, following the same lines.
\end{proof}

\bigskip

\subsection{Proof of Proposition~\ref{derivatives}}\hfill

\bigskip

We assume that $(N,K)$ is a pair in the second or third block of Table \ref{list}. The remaining cases are simpler and the modifications to the proof below are left to the reader. 

Notice that the spherical functions in $\Sigma_\cD^{\text{sing}}$ coincide with the spherical functions for the abelian pair $(\fv,K)$.
Moreover, for $F\in L^1(N)^K$ and $\xi_\cL=\rho_\fv(\omega)$, $\omega\in\fv$,
\begin{equation}\label{gelfandfourier}
\cG F(\xi_\cL,0,0)=\hat F(\omega,0)\ ,
\end{equation}
where $\hat F$ is the Fourier transform of $F$ on $\fv\oplus\fz$ (identified with $\fv^*\oplus\fz^*$ via the $K$-invariant scalar products).
 
Let $u(v)=\int_\fz F(v,z)\,dz\in\cS(\fv)^K$, and $g_{00}=\cE_K\hat u$. 
By Proposition \ref{hulanicki}, $F_{00}=g_{00}(\cL)\del_0\in\cS(N)^K$. By \eqref{gelfandfourier}, 
$$
\cG F_{00}(\xi_\cL,0,0)=g_{00}(\xi_\cL)=\cG F(\xi_\cL,0,0)\ ,
$$
for every $\xi_\cL\in\rho_\fv(\fv)$.
In particular,
$$
\int_\fz F_{00}(v,z)\,dz=u(v)\ .
$$

Let $R_0=F-F_{00}$. Since $\int_\fv F(v,z)\,dz=0$ for every $v$, 
Lemma~\ref{hadamard} implies that, with $m=\dim\fz$,
$$
R_0=\sum_{j=1}^m \de_{z_j}R'_j\ ,
$$
with $R'_j\in\cS(N)$ depending linearly and continuously on $R_0$ thus $F$. 

This proves the statement for $n=0$.

Suppose now that the statement is true for $n$, and consider the function $R_n$ in \eqref{geller}.
By \eqref{remainder}, $\widehat{R_n}(\omega,\zeta)\in \cS(\fn)^K$ vanishes with all derivatives up to order $n$ for $\zeta=0\in\fz$. 
By Lemma \ref{vanishing}, 
$$
\widehat{R_n}(\omega,\zeta)=\sum_{(j,k)\in E_n}\rho_{\fv,\fz}(\omega,\zeta)^j |\zeta|^{2k}h_{jk}\big(\rho(\omega,\zeta)\big)\ ,
$$
where the $h_{jk}\in\cS(\fn)^K$ depend linearly and continuously on $\widehat{R_n}$, thus on $F$.
 Undoing the Fourier transform,
\begin{equation}\label{non-constant}
R_n(v,z)=\sum_{(j,k)\in E_n}\rho_{\fv,\fz}(i\inv\nabla_\fv,i\inv\nabla_\fz)^j \Delta^k F_{jk}(v,z)
\end{equation}
where 
the $F_{jk}$ are in $\cS(N)^K$, depend linearly and continuously on $R_n$, and hence on $F$.

We replace now the constant coefficient operator $p(i\inv\nabla_\fv,i\inv\nabla_\fz)$ by $D$
using: 
\begin{lemma}
\label{replace_constant_coeff_op}
For $j\in \bN$, or $j\in \bN^2$ in the case of line 11, 
  $$
D^j-\rho_{\fv,\fz}(i\inv\nabla_\fv,i\inv\nabla_\fz)^j=
\sum_{|\alpha|=|j|+1} \partial_z^\alpha P_\alpha(v,\nabla_\fv,\nabla_\fz)
$$
where $P_\alpha(v,\nabla_\fv,\nabla_\fz)$ is a differential operator in $\fv$ and $\fz$ with polynomial coefficients in $v$. 
\end{lemma}

\begin{proof}
To fix the notation, we disregard the case of line 11. The proof for that case requires only minor modifications.
Fix an orthonormal coordinate system $(v_1,\dots,v_q)$ on $\fv$.
By \eqref{p}, 
$$
p(i\inv\nabla_\fv,i\inv\nabla_\fz)=\sum_{\ell=1}^mp_\ell(\de_{v_1},\dots,\de_{v_q})\de_{z_\ell}\ , \qquad D=\sum_{\ell=1}^m\la'( p_\ell)(V_1,\dots,V_q)\de_{z_\ell}\ ,
$$
where $\la'$ is the modified symmetrisation operator in \eqref{symm} and 
$$
V_r=\de_{v_r}+\sum_{\ell=1}^m \chi_{r\ell }(v)\de_{z_\ell}\ ,
$$
the $\chi_{r\ell}$ being linear functionals.

Each $p_\ell$ is a quadratic form on $\fv$ for the pairs in the second block, and linear for the pairs in the third block. We restrict ourselves to pairs in the second block, the other cases being even simpler.

With these assumptions, $D$ is a sum of terms, each being a product $V_rV_s\de_{z_\ell}$. We have

$$
V_rV_s=\de_{v_r}\de_{v_s}+\sum_{k=1}^mc_{rsk} \de_{z_k}+\sum_{k=1}^m\chi_{sk}(v) \de_{v_r}\de_{z_k}+\sum_{k=1}^m\chi_{rk}(v) \de_{v_s}\de_{z_k}+\sum_{k_1,k_2=1}^m\chi_{rk_1}(v)\chi_{sk_2}(v) \de_{z_{k_1}}\de_{z_{k_2}}\ .
$$

So, there are polynomials $Q_j$ such that
$$
V_rV_s\de_{z_\ell}=\de_{v_r}\de_{v_s}\de_{z_\ell}+\sum_{k=1}^m\de_{z_\ell}\de_{z_k}Q_{rsk}(v,\nabla_\fv,\nabla_\fz)\ .
$$

Summing over all terms, we have that
$$
D=p(i\inv\nabla_\fv,i\inv\nabla_\fz)+\sum_{k,\ell=1}^m\de_{z_k}\de_{z_\ell}P_{k\ell}(v,\nabla_\fv,\nabla_\fz)\ .
$$

It is then sufficient to use this expression of $D$ recursively.
\end{proof}

Using Lemma \ref{replace_constant_coeff_op}, 
we have:
$$
R_n=\sum_{j+2k=n+1}\frac1{j!k!}D^j\Delta^kF_{jk}+\sum_{|\al|=n+2}\de_z^\al R'_{\al}\ .
$$ 
Observing that each term in $D^j \Delta^k$ contains $n+1$ derivatives
and using the case $n=0$ for the $F_{j,k}$:
$$
F_{jk}=g_{jk}(\cL)\del_0+\sum_{\ell=1}^m\de_{z_\ell}R_\ell^{j,k}\ ,
$$
we have:
$$
R_n=\sum_{j+2k=n+1}\frac1{j!k!}D^j\Delta^kg_{j,k}(\cL)\del_0+\sum_{|\al|=n+2}\de_z^\al \tilde R_{\al}\ ,
$$ 
with  each $g_{jk}\in\cS(\bR^d)$ and $\tilde R_\al\in\cS(N)$ depending linearly and continuously on~$F$. 

This finishes the proof of Proposition~\ref{derivatives}.

\vskip1cm

\section{Proof of Theorem \ref{main}}\label{proof}

\bigskip

We denote by $d_\fv$ the number of invariants in $\rho_\fv$,
which is also the cardinality of $\cL$. Then $d=d_\fv+2$ for all cases except for line 11, where $d=d_\fv +3$.

For pairs in the first block of Table \ref{list} the variable $\xi_\cM$ must be disregarded.

Let $F\in\cS(N)^K$.
Take the functions $g_{jk}\in\cS(\bR^d)$, and $R_n$ of Proposition \ref{derivatives}.
Let $\psi\in \cS(\bR^d)$ be any function such that:
\begin{equation}
  \label{psi}
\de^j_{\xi_\cM}\de^k_{\xi_\Delta}\psi(\xi_\cL,0,0)=g_{jk}(\xi_\cL)\ ..  
\end{equation}
Fix $\ph(\xi_\cM,\xi_\Delta)$ smooth, with compact support and equal to 1 on a neighborhood of the origin. 
Then $\psi\ph\in\cS(\bR^d)$.

Define $G\in\cS(N)^K$ as
\begin{equation}\label{defG}
G=F-(\psi\ph)(\cL,\cM,\Delta)\del_0\ .
\end{equation}

We prove that from any such $\psi$
we can construct a Schwartz extension
of the Gelfand transform $\cG G$ of~$G$.

\begin{lemma} 
\label{G_H_al}
For every $n$, 
$G=\sum_{|\al|=n}\de_z^\al H_\al$
with $H_\al\in\cS(N)$, $|\al|=n$.
\end{lemma}

\begin{proof}
Fix $n\in\bN$. Let 
$$
T_n(\xi)=\sum_{|j|+2k\le n}\frac{\xi_\cM^j\xi_\Delta^k}{j!k!}g_{jk}(\xi_\cL)\ .
$$

Following \cite{ADR2}, Proposition 7.2, we modify $T_n$ outside of $\Sigma_\cD$ so to obtain a Schwartz function.
More precisely,
using \eqref{control_xi_j}, we define $\Omega_1$, resp. $\tilde \Omega_1$,  as the set of points $\xi\in\bR^d$ 
such that, for  $l=2,\ldots,d$, $|\xi_l|\leq 2C_l(1+|\xi_1|)^{\frac {\gamma_l}2}$,
resp. $|\xi_l|\leq 4C_l(1+|\xi_1|)^{\frac {\gamma_l}2}$,
and $\Omega_2$, resp. $\tilde \Omega_2$, as the tubular sets of points $\xi\in \bR^d$ 
such that  $\sum_{l=2}^d \xi_l^2\leq 1$, resp. $\sum_{l=2}^d \xi_l^2\leq 2$.
We set $\Omega=\Omega_1\cup \Omega_2$,
$\tilde\Omega=\tilde\Omega_1\cup \tilde \Omega_2$.

Fix $\chi\in C^\infty(\bR^d)$ such that $\chi=1$ on $\Omega$, $\chi=0$ on the complementary of $\tilde \Omega$;
as the distance between $\Omega$ and the complementary of $\tilde \Omega$ is positive,
we can assume that each derivative of $\chi$ is bounded.
Then $\chi T_n$ is a Schwartz function on $\bR^d$, which coincides with $T_n$ on a neighborhood of $\Sigma_\cD$.

We set:
$$
\psi_n(\xi)=\psi(\xi)-\chi(\xi) T_n(\xi)
\ ,\qquad
\bar\psi_n(\xi)=\big(1-\ph(\xi_\cM,\xi_\Delta)\big)\chi(\xi)T_n(\xi)-\psi_n(\xi)\ph(\xi_\cM,\xi_\Delta)
\ .
$$
$\psi_n$ and $\bar\psi_n$ are Schwartz functions of $\bR^d$.
As $\bar \psi_n$ coincides with $T_n-\psi$ on a neighbourhood of $(\xi_\cL,0,0)\in\bR^d$,
$\bar\psi_n$ vanishes with all derivatives in $\xi_\cM,\xi_\Delta$ up to order $n$ at each point $(\xi_\cL,0,0)\in\bR^d$.
Using Lemma~\ref{hadamard},
$$
\bar \psi_n(\xi)=\sum_{|j|+k=n+1}\xi_\cM^j\xi_\Delta^k \bar\psi_{jk}(\xi)\ ,
$$
with $\bar\psi_{jk}\in\cS(\bR^d)$.

Furthermore, for $\xi\in\Sigma_\cD$ we have:
$$
 \cG F(\xi)=\cG G(\xi)+(\psi\ph)(\xi) =T_n(\xi)+\cG R_n(\xi)\ ,
$$
therefore:
$$
\cG G(\xi)=\cG R_n(\xi)+ \bar \psi_n(\xi)\ .
$$
 Inverting the Gelfand transform,
$$
G=R_n+\sum_{|j|+k=n+1}\cM^j\Delta^k \bar\psi_{jk}(\cL,\cM,\Delta)\del_0\ ,
$$
where each term $\cM^j\Delta^k$ contains at least $n+1$ derivatives in the $\fz$-variables. 
\end{proof}

Denote by $\tilde G$ the Radon transform of $G$ on $N'$. 
Assume that the derivatives $\de_z^\al$ are referred to an orthonormal frame in $\fz$ having $\zeta_0$ as its last element. For every $n$, 
$$
\tilde G(v,t)=\de_t^n\tilde H_n(v,t)\ ,
$$
where $\tilde H_n$ is the Radon transform of $H_{0,\ldots,0,n}$ given by Lemma~\ref{G_H_al}.

This implies that $\cG'_0 \tilde G$ vanishes of infinite order at $\eta_d=0$.

\begin{lemma}
\label{lem_extension_der0_N'}
If $G$ is as above, 
there exists a Schwartz extension $g$ of $\cG'_0\tilde G$ 
satisfying:
\begin{enumerate}
\item all derivatives of $g$ vanish at all points of $\bR^d$ with $\eta_d=0$,
\item the support of $g$ is contained in a set where
 $|\eta_j|\leq C |\eta_1|^{\frac{\gamma'_j}2}$, $j=2,\ldots,d$,
\item for every $p\geq 0$, there exist a constant $C_p$ and an integer $q$ such that
$$
\|g\|_{(p)}\leq C_p\| \tilde G\|_{(q)}\ .
$$
\end{enumerate}
\end{lemma}

\begin{proof}
For pairs in the first and second block of Table \ref{list},
each group $N'$ is isomorphic to the Heisenberg group $H_m$
and Proposition 7.5 of \cite{ADR2} applies:
there exists a linear and continuous operator producing a Schwartz extension $g_0$ of $\cG'_0\tilde G$ 
such that  all of its derivatives vanish at all points of $\bR^d$ with $\eta_d=0$. 

For the third block,
each group $N'$ is isomorphic to the direct product of a Heisenberg group with an abelian factor.
Under these circumstances, 
the same conclusion can be drawn from a slight modifications of the arguments in 
Proposition 7.5 of \cite{ADR2}.

Using \eqref{control_eta_j}
we define $\Omega'$ and $\tilde\Omega'$ as the sets of points $(\eta_2,\ldots,\eta_d)\in\bR^{d-1}$ 
such that $|\eta_j|\leq 2C'_j$ and $|\eta_j|\leq 4C'_j$, $j=2,\ldots,d$, respectively.
We fix a function $\chi'\in C^\infty(\bR^{d-1})$ satisfying
$\chi'=1$ on $\Omega'$ and $\chi'=0$ on the complementary of $\tilde \Omega'$.
We set:
$$
g(\eta)=\left\{
  \begin{array}{ll}
g_0(\eta) \chi'(\eta_2 \eta_1^{-\frac{\gamma'_2}2}, \ldots,\eta_d \eta_1^{-\frac{\gamma'_d}2})
&\mbox{if}\, \eta_1\not=0\\
    0
&\mbox{if}\, \eta_1=0    
  \end{array}
\right.
$$
It is straightforward to show that $g$ satisfies the property (1), (2) and (3) of Lemma~\ref{lem_extension_der0_N'}.
\end{proof}

\begin{lemma}
\label{lem_extension_der0_N} 
If $G$ is as above, 
$\cG G$ admits a Schwartz extension $h$ 
such that for every $p\geq 0$, there are a constant $C_p$ and an integer $q$ for which:
\begin{equation}
  \label{h_p_G_q}
\|h\|_{(p)}\leq C_p\| G\|_{(q)}\ .
\end{equation}
\end{lemma}

\begin{proof}
Let $g$ be as in Lemma \ref{lem_extension_der0_N'}.
Then, for every $\eta\in\bR^d$, $\al\in\bN^d$, $\sigma,\tau\in\bN$, 
there is $q=q(\al,\sigma,\tau)$ such that, for every $\eta\in\bR$,
\begin{equation}\label{gestimates}
|\de^\al g(\eta)|\le C_{\al\sigma\tau}\|g\|_{(q)}\frac{|\eta_d|^\sigma}{(|\eta|+1)^\tau}\ .
\end{equation}

Using Proposition \ref{spherical_N_N'} and Lemma \ref{Sigmas},
$\cG'_0\tilde G$ and $\cG G$ are related by
$$
\cG G(\xi)=\cG'_0\tilde G\Big(\Theta(\xi)\Big)
\quad, \quad \xi \in \Sigma_\cD^{\text{\rm reg}}.
$$
Then the function
$$
h(\xi)=\begin{cases}g\Big(\Theta(\xi)\Big)&\text{ if }\xi_d>0\ ,\\
0&\text{otherwise.}
\end{cases}
$$
coincides with $\cG G$ on $\Sigma_\cD$.

Using \eqref{gestimates}
and the fact that, in all cases, $\eta=\Theta(\xi)$ satisfies $\eta_j=\xi_j \xi_d^{a_j}$ with $a_j\in \{0,-\frac12\}$ , 
it is easy to verify that, for every multi-index $\al$ and every $\sigma,\tau\in\bN$, there is $q'=q'(\al,\sigma,\tau)$ such that
$$
|\de^\al h(\xi)|\le C_{\al\sigma\tau}\|g\|_{(q')}\frac{|\xi_d|^\sigma}{(|\xi|+1)^\tau}\ ,
$$
for $\xi_d>0$. This shows that $h\in\cS(\bR^d)$ and that it depends linearly and continuously on $g$.
By (3) of Lemma~\ref{lem_extension_der0_N'} and the continuity
of the Radon transform,
 we obtain \eqref{h_p_G_q}. 
\end{proof}

We are now ready to prove \eqref{control}.
Given $p\geq 0$, take $q\ge p$ large enough so that  \eqref{h_p_G_q} holds.
By Proposition~\ref{hulanicki},
for any multiplier $m\in \cS(\bR^d)$,
the function $m(\cL,\cM,\Delta)\del_0$ is in $\cS(N)^K$ 
and there exists $r\geq p$ such that
\begin{equation}
  \label{cont_multiplier}
\big\|m(\cL,\cM,\Delta)\del_0\big\|_{(q)}\le C_q\|m\|_{(r)}\ .  
\end{equation}

We observe  now that,
following the same lines as in \cite{ADR2}, Proposition 7.4, 
the function $\psi$ can be chosen so that it satisfies \eqref{psi}
and 
$$
\|\psi\|_{(r)}\le C_r \|F\|_{(s)}\ ,
$$
for some $s\in\bN$.
By \eqref{defG} and \eqref{cont_multiplier},
$$
\|G\|_{(q)} \leq \|F\|_{(q)} + C_q \| \psi\phi\|_{(r)}
\leq C'_q\|F\|_{(q')} 
\ ,
$$ 
with $q'=\max(q,s)$.
By Lemma~\ref{lem_extension_der0_N},
$$
\|h\|_{(p)}\leq C_p\| G\|_{(q)}\leq C'_p \|F\|_{(q')}
\ .
$$

Therefore $u=h+\psi\ph$ is a Schwartz extension of $\cG F$,
and 
$$
\|u\|_{(p)}\leq 
\|h\|_{(p)} +C_p \|\psi\|_{(p)}
\leq 
C"_p\|F\|_{(q')}\ .
$$

This concludes the proof of Theorem \ref{main}.

\vskip1cm

\section{Appendix. Explicit invariants for pairs in Vinberg's list}\label{appendix}

%\bigskip

Assume that $(N,K)$ is a nilpotent  Gelfand pair and that $N$ is exactly step-two
i.e., non-commutative. As before, we decompose 
the Lie algebra $\gt n=\Lie N$ as a direct $K$-invariant sum 
$\gt n=\gt v\oplus\gt z$, where $\gt z$ is the centre of $\gt n$.
We assume that 
\begin{enumerate}
\item[(i)] $[\fn,\fn]=\fz$,
\item [(ii)] the action of $K$ on $\fv$ is irreducible, 
\item [(iii)] $\dim\fz>1$. 
\end{enumerate}
Then the pair $(K,\fn)$ is contained
in  Vinberg' table, see \cite[Table 3]{Y2}, which has 12 items.
Our main goal is to describe the algebra 
$\bR[\fn]^K$ of $K$-invariant polynomials on $\fn$.

\vskip1ex

\begin{lemma}\label{compl}
 Let $W$ be a finite-dimensional 
representation of $K$ defined over $\bR$. Let $W(\bC)$, $K(\bC)$ denote the complexifications of $W$ and $K$ respectively. Then
$\bC[W(\bC)]^{K(\bC)}=\bR[W]^K\otimes_{\bR}\bC$.
\end{lemma}
\begin{proof}
The statement easily follows from the fact that 
the action of $K$ on $W{\otimes}_{\rl}\cp$  commutes with 
the complex structure.
 \end{proof}

We say that some property holds ``for a generic point", if
it holds for all points of some non-empty open subset.  
Over an algebraically closed field, in our case $\bC$,
we will be using  a simplified version of the Igusa lemma,
see e.g.  \cite[Theorem~4.12]{VP}.  

\begin{lemma}\label{igusa}
Let $G$ be a complex algebraic group
and $V$ a vector space of its finite-dimensional linear 
representation.  Suppose that there are  algebraically independent 
polynomial functions 
$F_1,\ldots, F_m\in \cp[V]^G$ such that 
\begin{itemize}
\item[({\sf i})] \  for generic point 
$y=(y_1,\ldots,y_m)\in\cp^m$, the level subvariety
$\varphi^{-1}(y):=\{x\in V\mid  F_i(x)=y_i\  \text{ for }  i=1,\ldots,m\}$
contains a unique open $G$-orbit;
\item[({\sf ii})] \ the set of points $y\in\cp^m$ such that 
$\varphi^{-1}(y)\ne\varnothing$ contains 
a big open subset $U$, i.e., such that 
${\rm dim} (\cp^m\setminus U)\le m-2$.
%%%%%% ;   \item[({\sf iii})] \ $Y$ is normal. 
\end{itemize}
Then $\bC[V]^G=\bC[F_1,\ldots,F_m]$.
\end{lemma}

\begin{rmk}
Usually conditions of the Igusa lemma 
are stated in terms of a  morphism 
$\varphi:\,V\to\cp^m$ defined by 
$\varphi(v):=(F_1(v),\ldots,F_m(v))$. 
The level subvariety $\varphi^{-1}(y)$ is a fibre of this morphism. 
Condition $\varphi^{-1}(y)\ne\varnothing$ means that $y$ lies in the 
image of $\varphi$. There is also a third condition, 
which is automatically satisfied in our case.  
%%%% the affine space $\cp^m$ is  normal.  
\end{rmk}

\begin{rmk} 
Suppose we are in the setting of Lemma~\ref{compl}. 
Since all orbits of the compact group $K$ are closed,  
generic $K(\bC)$-orbits on $V=W(\cp)$ are closed as well 
and are separated by the polynomial invariants. 
This means that condition ({\sf i}) of Lemma~\ref{igusa} is satisfied 
if and only if  for generic $y\in\cp^m$ the level subvariety 
$\varphi^{-1}(y)\subset V$   is a $K(\cp)$-orbit.   
\end{rmk}

\begin{definition}  A finitely generated commutative associative algebra 
$A$ is said to be {\it free} if it admits a system of 
algebraically independent generators. Such a system is called 
a {\it free generating set}. 
\end{definition}

Suppose that  
$(N,K)$ is a nilpotent Gelfand pair, where 
$\gt n=\gt v\oplus\rl$ is a Heisenberg Lie algebra. 
Then $V:=\gt v(\cp)$ decomposes as a sum of 
two complex $K(\cp)$-representations, more precisely 
$V=W\oplus W^*$. The group $K(\cp)$ acts on
$\cp[W]$ and this action is {\it multiplicity free}, 
see \cite{BJR90}. We can express this statement  as 
$$
\cp[W]=\bigoplus\limits_{\lambda\in\Lambda(W)} V_\lambda, 
$$ 
where $V_\lambda$ is an irreducible representation with the highest 
weight $\lambda$ and $\Lambda(W)$ is the monoid of all
appearing highest weights. For more details see
\cite{BR00} or \cite{knop-mf}. The polynomial $K(\cp)$-invariants
on $V$ are described as follows 
$\cp[V]^{K(\cp)}=\bigoplus\limits_{\lambda\in\Lambda(W)} 
   (V_\lambda\otimes V_\lambda^*)^{K(\cp)}$,
where each subspace   $(V_\lambda\otimes V_\lambda^*)^{K(\cp)}$ 
is one-dimensional.  It is know that the monoid 
$\Lambda(W)$ is free and its generators are given in
\cite{BR00} or \cite{knop-mf}.  
If $\lambda_1,\lambda_2,\ldots,\lambda_r$ are free generators of 
$\Lambda(W)$, then $F_i:=(V_{\lambda_i}\otimes V_{\lambda_i}^*)^{K(\cp)}$
(with $i=1,2,\ldots,r$) are algebraically independent generators 
of  $\cp[V]^{K(\cp)}$.

Other rich sources of invariants are papers of G.\,Schwarz~\cite{gerry}
 and by  O.M.\,Adamovich and E.O.\,Golovina~\cite{devochki},
 where classification of the  representations of complex simple algebraic groups with 
 free algebras of invariants is carried out. In those papers the generating invariants are 
 described in terms of some representations and their highest weights, 
 which may not be very explicit.  On the other hand,  the degrees
 of generating invariants are also given and this can be enough in many cases 
 in view of the following fact.  
  
\begin{proposition}\label{degrees}
Let $A$ be an algebra of real or complex polynomials  
equipped with an algebraic grade-preserving action of some algebraic group 
$G$. Suppose that $A^G$
is generated by homogeneous algebraically independent 
polynomilas $a_1,\ldots,a_m$ and that
$b_1,\ldots,b_m\in A^G$ are algebraically
independent with $\deg b_i=\deg a_i$. Then 
$A$ is generated by $b_1,\ldots,b_m$.
\end{proposition}
\begin{proof}
The explanation is that 
$A^G$ and its subalgebra $B$, generated by $b_i$, 
have the same Poincar\'e series. 
To make it more explicit, let us compute the dimension 
of a linear space $A_d^G$ of the $G$-invariant polynomials 
of degree $d$. Since $a_i$ are algebraically independent 
generators, it is equal to 
the number of $m$-tuples 
$(t_1,\ldots,t_m)\in\mathbb Z_{\ge 0}^m$ such that 
$\sum\limits_{i=1}^{m} t_i\deg a_i=d$. 
Note that $\dim B\cap A_d$ is the same.  
\end{proof}

In practice it is not very difficult to verify that some polynomials are 
algebraically independent. Usually we do it by restricting them to
a suitable subspace. 
%%% In some cases verification is left to the reader. 

Trivial $K$-representations appearing in $\gt z$ 
correspond to linear $K$-invariants and they are 
often omitted from 
consideration. Suppose that $\gt w\subset\gt n$ 
is a $K$-invariant subspace. Then
$\rl[\gt n]\cong \cS(\gt n)$ has a $K$-invariant bi-grading
with respect to a $K$-invariant decomposition 
$\gt n=\gt w\oplus\gt m$, i.e., the grading components 
are $S^p(\gt w)S^q(\gt m)$.   
We have 
$\rl[\gt w]^K\subset\rl[\gt n]^K$ and for any 
generating set $\{F_1,\ldots,F_m\}$ in $\rl[\gt n]^K$,
the homogeneous components of the $F_i$'s of the bi-degrees $(\deg F_i,0)$  
form a set of generators in $\rl[\gt w]^K$. 
For that reason we do not 
state the answer for possible central reductions of $\gt n$, i.e., 
for pairs $(K,\gt n/\gt z_0)$, where $\gt z_0\subset\gt z$ is  a $K$-invariant 
subspace.  (Note that the representation of $K$ on $\gt n/\gt z_0$ 
is isomorphic to the representation of $K$ on the $K$-invariant 
complement of $\gt z_0$ in $\gt n$.)  

For a skew-symmetric matrix $x$, let 
${\rm Pf}(x)$ denote the Pfaffian of $x$. 
Having an $n{\times}n$ skew-symmetric matrix $z$ and 
an $n$-vector  $v$, we build a new skew-symmetric 
matrix 
$$
(z|v)=\begin{pmatrix}z&v\\-\trans v&0\end{pmatrix}.
$$
We let $z$ denote an element of $\gt z$ and 
$v$ an element of $\gt v$.  The symbols 
$|v|$ and $|z|$ stand for the norms on $\gt v$ and $\gt z$. 
Set $v^*:=\trans \bar v$, where 
$\bar{\phantom{v}}$ stands either for complex (or 
quaternionic, in cases 7 and 10) conjugation. 
Our notation in case 12 needs an explanation. Here 
$\gt v=\cp^2{\otimes}_{\cp}\qv^n$. Each vector 
$v\in\gt v$ can be presented as $v=w_1{\otimes}v_1+w_2{\otimes}v_2$,
where $w_1,w_2\in\cp^2$ are column vectors and 
$v_1,v_2\in\qv^n$ are row vectors. 
Set $v^*:=\trans\bar w_1{\otimes}\trans\bar v_1+\trans\bar w_2{\otimes}\trans\bar v_2$. Then $vv^*$ can be identified with a $2{\times}2$ quaternionic matrix. 
In this terms  a $K$-invariant real-valued scalar product on $\gt v$
can be written as $\tr(vv^*)$ or as $\tr(v^*v)$.  

The cases are numbered according to the lines of 
Vinberg's table, \cite[Table 3]{Y2}. We only unify items 
$4$ and $5$.

\begin{theorem}\label{generators} 
Generators $F_j\in\rl[\gt n]^K$ can be chosen as follows.\\[0.3ex]
\hbox  to \textwidth{
$
\begin{array}{l}
 K={\rm SO}_n, \gt n=\rl^n\oplus\gt{so}_n, n=2l: \\
F_j: \qquad \qquad  \tr(z^{2k})\ ,\ 1\le k \le l-1\ ,\qquad  {\rm Pf}(z)\ , 
 \qquad \trans v z^{2k}v\ ,\ 0\le k \le l-1\ ; \\
  K={\rm SO}_n, \gt n=\rl^n\oplus\gt{so}_n, n=2l+1: \\
F_j: \qquad \qquad  \tr(z^{2k})\ ,\ 1\le k\le l\ ,\qquad \trans v z^{2k}v\ ,\ 0\le k\le l-1\ ,\qquad
 {\rm Pf}(z|v); \\ 
K={\rm O}_n, \gt n=\rl^n\oplus\gt{so}_n: \\
F_j: \qquad \qquad \tr(z^{2k})\ ,\ 2\le 2k\le n\ ,\qquad\trans v z^{2k}v\ ,\ 0\le 2k \le n-1\ . \\
\end{array}
$
\hfil $(1)$}
\vskip0.3ex
  
\hrule
\vskip0.3ex

\noindent
\hbox  to \textwidth{
$ \begin{array}{l}
K=\Spin_7, \gt n=\rl^8\oplus\rl^7: \\
\qquad \qquad \ F_1=|v|^2, F_2=|z|^2.  \\
\end{array} $
\hfil $(2)$}
\vskip0.3ex

\hrule
\vskip0.3ex

\noindent
\hbox  to \textwidth{
$
\begin{array}{l}
K=G_2, \gt n=\rl^7\oplus\rl^7: \\
\qquad \qquad F_1=|v|^2, F_2=\trans vz, F_3=|z|^2.\\
\end{array} 
$
\hfil $(3)$}
\vskip0.3ex

\hrule
\vskip0.3ex

\noindent
\hbox  to \textwidth{
$
\begin{array}{l}
K={\rm U}_n, \fn=\bC^n\oplus\Lambda^2\bC^n, n\ge 3 \ : \\
F_j: \qquad  \tr((\bar z z)^k)\ ,\  2\le 2k\le n\ ,   \quad
 v^*(z\bar z)^{k}v\ ,\  0\le 2k < n\ ; \\
K={\rm SU}_n, \fn=\bC^n\oplus\Lambda^2\bC^n, n=2l, l>1\ : \\
F_j: \qquad  \tr((\bar z z)^k)\ ,\  1\le k< l\ , \quad   \RE{\rm Pf}(z),\, 
\IM{\rm Pf}(z)\ , \quad
 v^*(z\bar z)^{k}v\ ,\  0\le 2k < n\ ; \\
K={\rm SU}_n, \fn=\bC^n\oplus\Lambda^2\bC^n, n=2l+1, l\ge 1\ : \\
F_j: \enskip  \tr((\bar z z)^k)\ ,\  2\le 2k\le n\ ,   \enskip
 v^*(z\bar z)^{k}v\ ,\  0\le k < l\ , \enskip \RE{\rm Pf}(z|v),\, \IM{\rm Pf}(z|v)\ ; \\
 K={\rm U}_2, \fn=\bC^2\oplus\Lambda^2\bC^2\ : \\
\qquad\qquad  F_1=|v|^2, F_2=|z|^2; \\
K={\rm SU}_2, \fn=\bC^2\oplus\Lambda^2\bC^2,\ : \\
 \qquad\qquad  F_1=|v|^2, F_2=\RE z, F_3=\IM z; \\
% K={\rm U}_1, \fn=\bC\oplus\rl\ : \\
% \qquad\qquad F_1=|v|, F_2=z. \\
\end{array} $
\hfil $(4,5)$}  
%%(4,5)
\vskip0.3ex

\hrule
\vskip0.3ex

\noindent
\hbox  to \textwidth{
$
\begin{array}{l}
K={\rm U}_n, \fn=\bC^n\oplus\fu_n\ : \\
F_j: \qquad\qquad  i^k\tr(z^k), \ k=1,\ldots,n\ , \qquad
     i^kv^*z^kv, \ k=0,\ldots,n-1\ . \\
\end{array} $
\hfil $(6)$}
\vskip0.3ex

\hrule
\vskip0.3ex

\noindent
\hbox  to \textwidth{
$
\begin{array}{l}
K=({\rm U}_1{\times}){\rm Sp}_n, 
\gt n=\qv^n\oplus(HS^2_0\qv^n\oplus\IM\qv)\ : \\
F_j: \enskip \tr(z^k)\ , \ 2\le k\le n\ , \ v^*z^kv , 1\le k\le n-1\  , 
 \enskip a,b,c \ \text{ or } a, b^2+c^2, \\
\qquad \qquad \text{ depending on } K, \ \text{ where 
$z\in HS^2_0\qv$ and  $\{a,b,c\}$ is a basis of $\ \qv_0^*$.} \\     
\end{array} $
\hfil $(7)$}
%%(7)
\vskip0.3ex

\hrule
\vskip0.3ex

\noindent
\hbox  to \textwidth{
$
\begin{array}{l}
K={\rm U}_1{\times}\Spin_7, \gt n=\bC^8\oplus\bR^7\ : \\
\enskip F_1=|z|^2, \ F_2=|v|^2, \ F_3=\RE\big( z(v_1\bar v_2)\big), \
 F_4=|v_1|^2|v_2|^2-\big(\RE(v_1\bar v_2)\big)^2, \\ 
\text{ where $v=v_1+iv_2$ accordingly 
to a $\Spin_7$-invariant decomposition $\cp^8=\rl^8\oplus\rl^8$,}\\
\text{ each $\bR^8$ is identified with $\bO$ and $\bR^7$ with $\IM\bO$.} \\
\end{array} $
\hfil $(8)$}
%(8)
\vskip0.3ex

\hrule
\vskip0.3ex

\noindent
\hbox  to \textwidth{
$\begin{array}{l}
K=\Sp_1{\times}\Sp_n, \fn=\bH^n\oplus\fs\fp_1\ : \\
\qquad \qquad F_1=|v|^2, \ F_2=|z|^2. \\
\end{array}$
\hfil $(9)$}
\vskip0.3ex

\hrule
\vskip0.3ex

\noindent
\hbox  to \textwidth{
$
\begin{array}{l}
K=\Sp_2{\times}\Sp_n, \fn=\bH^2\otimes\bH^n\oplus\fs\fp_2, n\ge 2 \ : \\
F_j\ : \qquad \tr(z^2), \ \tr(z^4),  \ |v|^2, \ 
   \tr(zv(zv)^*), \ \tr((zvv^*-vv^*z)^2), \ \tr((vv^*)^2)\ ; \\
K=\Sp_2{\times}\Sp_1, \fn=\qv^2\oplus\gt{sp}_2 \ :\\
\enskip
    F_1=\tr(z^2), \ F_2=\tr(z^4), \ F_3=|v|^2, \ 
    \ F_4=\tr(zv(zv)^*), \ F_5=\tr((zvv^*-vv^*z)^2)\ .\\
\end{array}  $  
\hfil $(10)$}
%(10)
\vskip0.3ex

\hrule
\vskip0.3ex

\noindent
\hbox  to \textwidth{
$
\begin{array}{l}
K={\rm U}_2{\times} \SU_n, 
\fn=\bC^2\otimes\bC^n\oplus\fu_2, \ \text{ with $n\ge 2$} \   \\
(\text{in case $n\ge 3$, $K$ can be replaced by $\SU_2{\times}\SU_n$})\ :  \\ 
\enskip F_1=i\tr(z), \ F_2=\tr(z^2), \ F_3=|v|^2, \ F_4=\tr((vv^*)^2), \ F_5=iv^*zv. \\
\end{array} $
\hfil $(11)$}
%(11)
\vskip0.3ex

\hrule
\vskip0.3ex

\noindent
\hbox  to \textwidth{
$
\begin{array}{l}
K={\rm U}_2{\times}\Sp_n, \fn=\bC^2\otimes\bH^n\oplus\fu_2, \ \text{ with $n\ge 2$} \ :\\
\enskip F_1=i\tr(z), \, F_2=\tr(z^2), \, F_3=|v|^2, \, 
F_4=\tr((vv^*)^2), \,  
 F_5(x+iy)=|x|^2|y|^2-(\trans xy)^2,   \\  
\enskip  F_6=\tr(v^*izv)\,, 
 \text{ where $v=x+iy$ accordingly to} \\ 
 \quad \text{an }  \Sp_1{\times}\Sp_n-\text{invariant decomposition} \
  \cp^2\otimes\qv^n=\rl^{4n}{\oplus}\rl^{4n}. \\
\end{array} $
\hfil $(12)$}
%(12)
 \end{theorem}

\begin{corollary} If $(N,K)$ is a nilpotent Gelfand pair, the  $K$-action  on $\gt v$
is irreducible,  and
$\dim\gt z>1$, then the algebra $\rl[\gt n]^K$ of $K$-invariant
polynomials on $\gt n$  
is free.  
\end{corollary}

\begin{definition} Suppose that  a group $G$ acts linearly on a 
finite-dimensional vector 
space $V$. A subgroup $H\subset G$ 
is said to be a  {\it generic stabiliser} of this action if there 
is an open subset $U\subset V$ such that 
the stabiliser $G_v$ is conjugate to $H$ for all $v\in U$. 
It is usually denoted by $G_*(V)$. 
\end{definition}

By a deep result of Richardson, generic stabiliser exists for 
all linear actions of reductive 
groups, see e.g.~\cite[Theorem~7.2]{VP}. 
In the following $\zeta$ stands for a generic point 
of $\gt z$ (or $\gt z(\cp)$), $K_\zeta$ for its stabiliser 
in $K$ (up to a connected component), 
and $K_*(\gt n)$ for the generic stabiliser 
of the action of $K$ on $\gt n$. Note that
for the maximal number of algebraically independent 
$K$-invariants, called transcendence degree of 
$\rl[\gt n]^K$ and denoted by 
${\rm tr.deg}\,\rl[\gt n]^K$, we have the equality 
${\rm tr.deg}\,\rl[\gt n]^K=\dim\gt n-\dim K+\dim K_*(\gt n)$. 

Suppose that $\gt c\subset\gt z$ is an analogue of a Cartan subspace, i.e., 
it is a linear subspace containing a generic point 
$\zeta$ such that $K_\zeta$ stabilises also all points of 
$\gt c$ and for a generic point $\zeta'\in\gt c$ the intersection 
$K\zeta'\cap \gt c$ is finite. We do not address here the question of
existence of such $\gt c$. Note only that in the rank one case 
any line $\rl\zeta$ with $\zeta\in\gt z$ being generic satisfies both properties.    
Set $r:=\dim\gt c$. Then $r={\rm tr.deg}\,\rl[\gt z]^K$. 
Let $F_1,\ldots, F_m$ be  a free
generating set in $\rl[\gt n]^K$ with 
$\rl[F_1,\ldots,F_r]=\rl[\gt z]^K$. 

\begin{corollary} Suppose that $(N,K)$ is a nilpotent Gelfand pair 
with an irreducible action of $K$ on $\gt v$,
$\dim\gt z>1$, and $\gt c\subset\gt z$ is a 
linear subspace satisfying both assumptions of the previous paragraph.   
Then $\rl[\gt v\oplus\gt c]^{K_\zeta}=\rl[\tilde F_{r+1},\ldots, \tilde F_{m}]
 \otimes \rl[\gt c]$, where $\tilde F_j$ is the restriction of $F_j$ to 
 $\gt v\oplus\gt c$. In particular, 
 the algebra $\rl[\gt v\oplus\gt c]^{K_\zeta}$ is free and has the same number of generators, 
 $m$, as $\rl[\gt n]^K$. 
\end{corollary}
\begin{proof}
Consider the composition 
$\rl[\gt n]^K\to \rl[\gt v\oplus\gt c]^{K_\zeta} \to \rl[\gt v\times\{\zeta\}]^{K_\zeta}$
of two restriction morphisms. It is known to be surjective, see 
e.g. \cite[proof of Theorem~1.3]{Y2}. The polynomials
$F_1,\ldots,F_r$ are constant on $\gt v\times\{\zeta\}$. 
On the other hand ${\rm tr.deg}\,\rl[\gt v\times\{\zeta\}]^{K_\zeta}=m-r$, 
because a generic $K_\zeta$ orbit on $\gt v$ has codimension 
$m-r$. Hence $\rl[\gt v\times\{\zeta\}]^{K_\zeta}$ is generated 
by the restrictions of $F_j$ with $r<j\le m$ 
and, since the action of $K_\zeta$ on $\gt c$ is trivial, 
$\rl[\gt v\oplus\gt c]^{K_\zeta}=\rl[\tilde F_{r+1},\ldots, \tilde F_{m}]
\otimes \rl[\gt c]$.
\end{proof}

Let $\gt g$ be a complex simple Lie algebra of rank $r$ and 
$\alpha_1,\ldots,\alpha_r$ a set of simple roots with a standard 
numbering \cite[Tables]{VO}. 
We let $\varpi_i$ denote the $i$'th fundamental highest weight 
of $\gt g$ and $R(\varpi_i)$ 
the corresponding irreducible representation.   

\vskip1ex

\noindent
{\it Proof of Theorem~\ref{generators}.} \  We consider each case separately.
%%%% beginning with he first one.

\noindent
{\bf 1.} \  The complexification of $\gt n$ is isomorphic to 
$R(\varpi_1){\oplus} R(\varpi_2)$ as a representation of the complex 
group $\SO_n$.
This representation can be found in Table~3a, lines~2 and 5, 
\cite{gerry}.
The description of invariants depends on the parity of $n$.
Suppose first that $n=2l$, i.e., $K(\bC)$ is of type $D$ and we are 
looking at line~2 in \cite[Table~3a]{gerry}.  
The set $F_1,\ldots, F_m$ of Theorem~\ref{generators}(1) has the same 
degrees (and bi-degrees).  Due to Proposition~\ref{degrees},
it only remains to show that the proposed generators are
algebraically independent. 

As is well known, the first $l$ polynomials form a 
generating set in $\rl[\gt{so}_n]^{{\rm SO}_n}$. 
In particular, $F_1,\ldots,F_l$ are algebraically 
independent.  
Assume that there is a non-trivial 
equation $Q(F_1,\ldots,F_m)=0$. Since the first $l$ polynomials
are algebraically independent,  $Q$ depends non-trivially on the 
$F_j$ with $j>l$. Moreover, for generic $\zeta\in\gt z$,
the restriction to $V\times\{\zeta\}$ gives us a non-trivial equation on the 
$\tilde F_j(v)=F_j(v,\zeta)$ with $j>l$. 
Suppose w.l.o.g. that $\zeta$ is a block-diagonal matrix
$$
\diag\left(\begin{pmatrix}0&t_1\\-t_1&0\end{pmatrix},\begin{pmatrix}0&t_2\\-t_2&0\end{pmatrix},\ldots,\begin{pmatrix}0&t_l\\-t_l&0\end{pmatrix}\right)\ ,
$$ 
with pairwise 
distinct $|t_i|$. Then $\tilde F_{l+k+1}=
t_1^{2k}(x_1^2+x_2^2)+t_2^{2k}(x_3^2+x_4^2)+\ldots+ t_{l}^{2k}(x_{n-1}^2+x_n^2)$, 
with $0\le k\le l-1$ and $x_1,\ldots,x_n$ being coordinates on $\gt v$. 
Because all $t_i^2$ are distinct, %%% these expressions are linearly independent.
the polynomials $\tilde F_j$ are linear independent, hence,
algebraically independent and $Q$ must have been trivial. 

Suppose now that $n=2l+1$. Then $K(\cp)$ is of type $B$
and we are looking  at  line~2 in \cite[Table~3a]{gerry}.  
The set $F_1,\ldots, F_m$ of Theorem~\ref{generators}(1) has the same 
degrees (and bi-degrees). 
%%%%%
\begin{comment}
Because the last polynomial 
$F_m={\rm Pf}(z|v)$ is less trivial, we give some explanations, why 
it is an invariant. 

The determinant
of the skew-symmetric matrix $(z|v)$ has degree $n-1$ in variables $z$ and $2$ in $v$. 
It is invariants under the adjoint action of $\SO_{n+1}$ and 
after the arising action of its subgroup
${\rm SO}_n{\times}\{1\}\subset {\rm SO}_{n+1}$. 
This action of $\SO_n$ coincides with 
the action of $K$ on $\gt n$. Since $K$ is connected,
a square root of the determinant is also $K$-invariant.   
\end{comment}
%%% 
%%% The rest of the 
The proof goes as in the case of even $n$, with the 
single difference that  $\tilde F_{l+k+1}$
is equal to 
$t_1^{2k}(x_1^2+x_2^2)+t_2^{2k}(x_3^2+x_4^2)+\ldots+ t_{l}^{2k}(x_{n-2}^2+x_{n-1}^2)$
for $0\le k\le l-1$ (the sums are linearly independent), and 
$\tilde F_m=t_1t_2\ldots t_l x_n$. Clearly, these $l+1$ polynomials 
are algebraically independent.  Hence the polynomials 
$F_1,\ldots,F_m$ are algebraically independent. 
 
In case $K={\rm O}_n$ some modification are needed. 
An additional element  acts on most of the 
${\rm SO}_n$-invariants trivially. The only exceptions are 
Pfaffians, which must be replaced by the determinants, or, equivalently, 
by $\tr(z^n)$ in case $n$ is even and by $\trans vz^{n-1}v$ in case $n$ is odd.  

\medskip

\noindent
{\bf 2.} \ This is a case of the so-called double transitivity, 
see \cite{Ko}, the group $K$ acts transitively on each product of 
two spheres, one  in $\gt v$ and one in $\gt z$.  We  can also apply 
Lemma~\ref{igusa}. For non-zero $a,b\in\cp$, the level subvariety 
$X_{a,b}=\{x\in\gt n(\cp) | F_1(x)=a, F_2(x)=b\}$ is a single $K(\cp)$-orbit, 
because the stabiliser $K(\cp)_v=G_2$ of a generic 
$v\in\cp^8$ acts transitively on complex spheres in $\cp^7$. 
It is also quite clear that $X_{a,b}\ne\varnothing$ for 
all $a,b\in\cp^2$. 

%%% The invariants 
%%% $F_1=|v|^2$ and $F_2=|z|^2$ are algebraically independent 
%%% and clearly the image of $\fn(\bC)$ under $\varphi$
%%%is the whole $\bC^2$. For any $v\in V(\bC)$ with 
%%$|v|^2\ne 0$ the 
%%stabiliser $K(\bC)_v=G_2$ acts transitively on spheres 
%%in $\fz(\bC)$. Thus generic fibres of $\varphi$ 
%%are $K(\bC)$-orbits.
%%\hfill{$\Box$} 

\medskip

\noindent 
{\bf 3.} \ In this case we use Lemma~\ref{igusa}. 
Fix a non-zero value of $F_3$, let us say $F_3(\zeta)$, 
$\zeta\in\gt z(\cp)$.
The complex group $G_2$ acts transitively on the complex 
sphere $F_3(z)=F_3(\zeta)$ in $\fz(\bC)$.  Therefore 
condition ({\sf i}) of Lemma~\ref{igusa} is fulfilled if and only if 
the stabiliser $K(\bC)_\zeta=\SL_3$ acts transitively on 
$X_\zeta:=\{v+\zeta\mid F_1(v)=a, F_2(v+\zeta)=b\}$ with $a,b\in\cp$. 
As a representation of $\SL_3$ we can decompose $\gt v(\cp)=\cp\oplus\cp^6$, 
where $\cp^6$ is the orthogonal complement to $\zeta$. Now
$X_\zeta=\{w\}\times\{v\in\bC^6\mid F_1(v)=a-F_1(w)\}$,
where $\trans w\zeta=b$. Such vector $w$ is unique. 
The stabiliser, $\SL_3$ acts on $\bC w$ trivially, but in the 
orthogonal complement $\bC^6=\cp^3{\oplus}(\cp^3)^*$ transitively on 
complex spheres $F_1(v)=c$, with $c\ne 0$. 
For a generic triple $y=(a,b,F_3(\zeta))\in\cp^3$, where 
$a\ne F_1(w)$, the group $\SL_3$ acts on  $X_\zeta$
transitively, hence,  $\varphi^{-1}(y)$ is a single $K(\bC)$-orbits.
Clearly $\varphi^{-1}(y)$ is non-empty for all 
$y\in\cp^3$. Thereby condition ({\sf ii}) is satisfied as well. 

\medskip 

\noindent 
{\bf 4,5.} \  At first assume that $K={\rm U}_n$ and $n\ge 4$. 
Here we use the general theory developed for the multiplicity free 
actions, see beginning  of this section. 
The space $\gt n$ carries a $K$-invariant complex 
structure and the action of $K$ on $\cp[\gt n]$ is multiplicity free,
i.e., each irreducible representation appears only once
$\cp[\gt n]=\bigoplus\limits_{\lambda\in\Lambda(\gt n)} V_\lambda$.
The highest weights semigroup $\Lambda(\gt n)$ is 
a free monoid, generated by weights $\lambda_1,\ldots,\lambda_m$.
The algebra $\cp[\gt n(\cp)]^{K(\cp)}$ is freely generated 
by the invariants sitting in $V_{\lambda_i}\otimes V_{\lambda_i}^*$.
Descriptions of 
$\Lambda(\gt n)$ can be found in \cite{knop-mf}
or \cite{BR00}. 
In our case $m=n$ and this description leads to the set of generators 
given in Theorem~\ref{generators}(4,5). Another possible approach could 
be to take only the degrees of generating invariants from \cite{knop-mf} 
and then prove that the $F_j$'s are algebraically independent. 

The case $K=\SU_n$ has one additional complication. 
%%%%% The algebra $\cp[\gt n]^K$ %%% is non-trivial
There are non-constant $K$-invariant complex polynomials on 
the complex space $\gt n$, e.g., for even 
$n$, the determinant of $z$, for odd $n$,
the determinant of a skew-symmetric $(n{+}1){\times}(n{+}1)$
matrix $(z|v)$. 
%%% constructed from $z$ and $v$.\footnote{At this stage, I suppose that we can just say `the %%% determinant of $(z|v)$'} 
This does not spoil the situation very much, since 
the algebra $\bR[\fn]^K$ remains free according 
to \cite[Table 1a, line 11]{gerry}. Over real 
numbers the additional invariants are the real and 
imaginary parts of ${\rm Pf}(z)$ or ${\rm Pf}(z|v)$, depending 
on the parity of $n$.  Their product, which is 
either $\tr((\bar z z)^l)$ or $v^*(z\bar z)^lv$,
should be removed from the set of generators.
 
Consider now small $n$.  
For $n=1$, the situation is trivial:
$\gt n$ is commutative or a Heisenberg Lie algebra.
These two cases are omitted from consideration. 
For $n=2$, the centre $\gt z$ is $\cp$ and 
$\rl[\gt n]^K$ is either $\rl[|v|^2,|z|^2]$ or 
$\rl[|v|^2,\gt z]$. 

The last case is $n=3$.  Here the 
invariants are the same as for $n\ge 4$. 
There is only one simplification, if
$K=\SU_3$, $\gt n$ can be identified with 
$\cp^3\oplus\cp^3$ and ${\rm Pf}(z|v)$ is proportional  to 
$\trans vz$. Thus the set of generating invariants is
$|v|^2,|z|^2$, ${\rm Im}(\trans vz)$,
${\rm Re}(\trans vz)$.    

\medskip

\noindent 
{\bf 6.} \ The first $n$ invariants form a generating 
set in $\rl[\gt{u}_n]^{{\rm U}_n}$ or, after complexification,  
in $\cp[\gt{gl}_n]^{\GL_n}$. Several other facts about these 
polynomials are well known. 
For example, for generic $s\in\cp^n$ the level subvariety  
$\{\xi\in\gt{gl}_n \mid F_i(\xi)=s_i, i=1,\ldots,n\}$ is a single 
$\GL_n$-orbit. In order to show that the proposed invariants  
$F_1,\ldots,F_m$, with $m=2n$, generate 
$\rl[\gt n]^K$ we check two conditions of 
Lemma~\ref{igusa}.  

Let $t\in\fg\fl_n$ be a regular semisimple 
element, i.e., a diagonal matrix with pairwise distinct 
entries $(t_1,\ldots,t_n)$. Set $T=(\GL_n)_t$.
Then $T=(\cp^{^\times}\!)^n$ is a maximal torus 
in $\GL_n$.  We have $\gt v(\cp)=\cp^n\oplus(\cp^n)^*$.
Over $\cp$ the polynomials $v^*z^kv$ are proportional to  
$y(z^kx)$, where $y\in(\cp^n)^*$, $x\in\cp^n$. 
On 
$V(\bC)\times\{t\}$ the polynomials $F_{n+1},\dots,F_{m}$  
reduce to $\sum\limits_{i=1}^n t_i^k x_iy_i$ 
with $k=0,\ldots,n-1$, where 
$x_i$ and $y_i$ are suitable coordinates on $\cp^n$ and $(\cp^n)^*$.
Since the $t_i$ are pairwise 
distinct, the Vandermonde determinant is non-zero 
and the sums (restrictions of invariants) 
span the linear space generated by $x_iy_i$.
Therefore,  the subvariety defined by  $x_iy_i=c_i\ne 0$, 
is  a single $T$-orbit. Hence, for 
$c=(c_1,\ldots,c_m)\in\cp^m$ 
with $c_i=F_i(t)$ for $1\le i\le n$ and generic 
$c_j$ with $j>n$, the level subvariety 
$\varphi^{-1}(c)=\{v\in\gt v(\cp) \mid F_i(v)=c_i\}$, 
is a single $\GL_n$-orbit.  
Thus condition $({\sf i})$ of Lemma~\ref{igusa} is satisfied. 

Each point in $\cp^n$ 
can be presented as $(\tr(\xi),\tr(\xi^2),\ldots,\tr(\xi^n))$, 
where $\xi\in\gt{gl}_n$ is a regular element, i.e., such that 
$\dim(\GL_n)_\xi=n$. 
The invariants $F_i$ with $i>n$ on $\gt v(\cp)\times\{\xi\}$
reduce to the pairings $(\cp^n)^*\times\cp^n\to\cp$ given
by the matrices ${\rm I}_n$, $\xi$, $\xi^2,\ldots,\xi^{n-1}$. 
Since the elements ${\rm I}_n,\xi,\ldots \xi^{n-1}$ are
linear independent (that is the main property of 
a regular element), we get $n$ linear independent 
pairings.  Hence 
for every $c\in\cp^n$ the set 
$\{v\in\gt v(\cp)\mid F_j(v+\xi)=c_{j-n}, n+1\le j\le 2n\}$
is non-empty. Therefore the set 
$\varphi^{-1}(c)$ is non-empty for all $c\in\cp^m$ and 
condition $({\sf ii})$ of Lemma~\ref{igusa} is satisfied as well.
   
\medskip

\noindent 
{\bf 7.} \ Let $z$ be an element in $HS^2_0\qv^n$ 
and $a,b,c$ coordinates on $\bH_0$ with $a$ being ${\rm U}_1$-invariant. 
First we look on $\Sp_n$-invariants in 
$\rl[\qv^n{\oplus}HS^2_0\qv^n]$.
The complex group $\Sp_{2n}(\cp)$, which is the complexification of 
$\Sp_n$, acts on $\gt n(\cp)$ via 
$R(\varpi_1)\oplus R(\varpi_2)$. This sum appears as a 
subrepresentation in
\cite[Table~4a, item~2]{gerry} and the invariants are described 
in lines 1,\,2 of \cite[Table~4b]{gerry}. The representation  
$R(\varpi_2)$ is isomorphic to a $\gt g_0$ action on a subspace 
$\gt g_1$ arising from the 
symmetric decomposition ($\mathbb Z_2$-grading)
$\gt g=\gt g_0\oplus\gt g_1$ with $\gt g=\gt{sl}_{2n}$,
$\gt g_0=\gt{sp}_{2n}(\cp)$. The generating $\Sp_{2n}(\cp)$-invariants on 
$\gt g_1$ are
known to be the traces of powers, $\tr(z^k)$, with $k=2,\ldots,n$. 
On $\qv^n$ there is only one generator, 
$|v|^2$. The ``mixed" case is dealt with in 
item~2 of \cite[Table~4b]{gerry}. 
The arising invariants have degrees $2$ in $v$ and 
$1,\ldots,n-1$ in $z$.
%%% and for each bi-degree there is a unique choice. 
%%Over $\rl$ these invariants are  
We have to show that 
$v^*z^kv$ is not proportional to $|v|^2\tr(z^k)$
for $k=2,\ldots,n-1$. Indeed 
if $\zeta\in HS^2_0\qv^n$ is a diagonal matrix $(t_1,\ldots,t_n)$ 
with pairwise distinct $t_i\in\rl$ and $\sum_{i=1}^{n}t_i=0$, 
then $v^*\zeta^kv=\sum_{i=1}^{n} t_i^k|v_i|^2$ and the sum is not 
proportional to 
$|v|^2\tr(\zeta^k)=(\sum_{i=1}^{n}t_i^k)(\sum_{i=1}|v_i|^2)$.   

The groups $\Sp_n$ and ${\rm U}_1{\times}\Sp_n$ have 
the same orbits on $\bH^n\oplus HS_0^2\qv^n$ 
and hence the same invariants. In case $K=\Sp_n$ we add 
to the above polynomials all the coordinates $a,b,c$, and  
if $K={\rm U}_1{\times}\Sp_n$, then $a$ and $b^2+c^2$. 

\medskip

\noindent 
{\bf 8.} \  Here $\gt n(\cp)$  
decomposes as $\cp^7\oplus\cp^8\oplus\cp^8$. 
As a representation of $\Spin_7$ it is 
$R(\varpi_1)\oplus 2R(\varpi_3)$ 
and is a subrepresentation of both items 10 and 11 of \cite[Table~3a]{gerry}.
In particular, the algebra of $\Spin_7$-invariants is free. 
It is not difficult to tell which invariants live 
on our subrepresentation. There are five generators.
Identifying $\cp[\gt n(\cp)]$ and $S(\gt n(\cp))$, we can say that  
$F_1\in S^2(\cp^7)$, $F_4'$, $F_5'$ sit in two 
different copies of 
$S^2(\cp^8)$ (one in each), $F_2\in\cp^8{\otimes}\cp^8$,
and  $F_3\in\cp^7\otimes\cp^8{\otimes}\cp^8$.  
The non-trivial action of $\cp^*$
replaces $F_4'$ and $F_5'$ by $F_4=F_4'F_5'$. 

In order to express $F_3$ and $F_4$ over $\rl$, we 
decompose $\cp^8=\rl^8\oplus\rl^8$ as a sum of two 
real $\Spin_7$-invariant subspaces and identify each of them as the space $\bO$ of octonions. We also identify $\fz=\bR^7$ with $\IM\bO$. Then we write  
each $v\in\gt v$ as an octonionic vector $\begin{pmatrix}v_1\\v_2\end{pmatrix}$,
with ${\rm U}_1$ acting by ${\SO}_2$-matrices.
Now 
$F_3(v,z)=\RE\big( z(v_1\bar v_2)\big)$ and 
$F_4(v)=|v_1|^2|v_2|^2-\big(\RE(v_1\bar v_2)\big)^2$.

\medskip  
  
\noindent 
{\bf 9.} \  The group $\Sp_n$ acts transitively on the spheres in 
$\qv^n=\cp^{2n}$. Thus $\Sp_n$ invariants in this 
case are generated by $|v|^2$ 
and linear functions on $\IM\qv$. Clearly 
$|v|^2$ is also $\Sp_1$-invariant. 
Hence $\rl[\gt n]^K=\rl[|v|^2,|z|^2]$. 
This is another case of the double transitivity \cite{Ko}.

\medskip

\noindent 
{\bf 10.} \ First of all, $K_*(\gt n)$ is trivial in case $n=1$ and 
$K_*(\gt n)=\Sp_{n-2}$ for $n>1$;
$K_\zeta={\rm U}_1{\times}{\rm U}_1{\times}\Sp_n$. 
Therefore the  codimension of a generic 
$K$-orbit is equal to $8n+10-(8n+4)=6$ in case 
$n>1$ and to $5$, when $n=1$. 
On $\gt z$ there are two generating invariants, $F_1=\tr(z^2)$, $F_2=\tr(z^4)$.
The representation of $K$ on $\gt v$ 
is a real form of $\gt g_1$ arising from the symmetric pair 
$(\gt g,\gt g_0)=(\gt{sp}_{2n+4},\gt{sp}_{2n}{\oplus}\gt{sp}_4)$,
which is of rank $2$, if $n>1$, and of rank $1$, if $n=1$.
There are two generators: $F_3=|v|^2$ and $F_6$.
Over $\cp$, the polynomial $F_6$ is 
the restriction to $\gt g_1$ of $\tr(A^4)$,
where $A$ is a matrix in $\gt{sp}_{2n+4}$.  Over $\rl$, 
and only in case $n>1$, we write $F_6$ as  
$\tr((vv^*)^2)$, where $v$ is identified with a $2{\times}n$ 
quaternionic matrix.

For the description of the mixed invariants $F_4,F_5$ some other 
arguments are needed.
Note that $\gt v(\cp)=\cp^4\otimes\cp^{2n}$ and 
as a representation of the complex group $\Sp_{2n}$ 
it is a sum of four copies of the defining representation 
$\cp^{2n}$. According to the First Fundamental Theorem of the 
so called ``classical invariant theory", cf. \cite[Subsection~9.3]{VP} , 
the algebra $\cp[\gt v(\cp)]^{\Sp_{2n}}$ has six generators 
$H_1,\ldots H_6$, all of them of degree $2$. These generators 
correspond to non-degenerate pairings between different copies of $\cp^{2n}$.  
The second (complex) group $\Sp_4$ acts on the linear space 
$\left<H_1,\ldots,H_6\right>_\cp$ generated by 
the $H_i$. In order to understand this action, we compute
$S^2(\cp^4\otimes\cp^{2n})=S^2\cp^4\otimes S^2\cp^{2n}\oplus
\Lambda^2\cp^4\otimes\Lambda^2\cp^{2n}$. 
The invariants $H_i$ sit in $\Lambda^2\cp^{2n}$. 
Hence the action of $\Sp_4$ on $\left<H_1,\ldots,H_6\right>_\cp$ 
decomposes as a sum 
$\cp\oplus\cp^5$ of the trivial and the irreducible $5$-dimensional representation
$R(\varpi_2)$.  The trivial representation corresponds to 
the $K$-invariant $F_3=|v|^2$. 

Now we are in the situation of item~(1)  with 
the action of $\SO_5$ on $\rl^5\oplus\gt{so}_5$. 
There are six generating invariants. 
In case $n>1$ they  must be algebraically independent. 
In case $n=1$, the second invariant, of degree $2$, on $\rl^5$
is proportional to $F_3^2$. The other five are algebraically independent. 
The mixed invariants are of 
bi-degrees  $(1,2)$, $(2,2)$ in $H_i$ and $z$ respectively. 

Over $\rl$ the subspace generated by the $H_i$ can be identified 
with a hermitian $2{\times}2$-matrix $vv^*$. 
Then the  $\Sp_2$-invariant norm is defined
as $|vv^*|^2=\tr((vv^*)^2)$. The action of $\gt{sp}_2$ is 
$z{\cdot}vv^*=zvv^*-vv^*z$. Therefore, the invariant 
of be-degree $(2,2)$, which is $|z{\cdot}h|^2$ 
($h\in\left<H_1,\ldots,H_6\right>_\rl$), should be written 
as $\tr((zvv^*-vv^*z)^2)$. 
It only remains to show that  $F_4=\tr(zv(zv)^*)$ is not proportional to $F_1F_3$.
%%% ; and 
%%% that $F_5$ does not lie in a linear space generated by 
%%% $F_3F_4$, $F_1F_6$, and $F_1F_3^2$. 

%%% Due to Lemma~\ref{degrees}, if we prove that 
%%% $F_1,\ldots,F_6$ are algebraically independent, the case will be 
%%% finished. Obviously $\tr(z^2)$ and $\tr(z^4)$ are algebraically independent. 

Chose a generic $\zeta\in\gt z$. We may assume that 
it is given by a pair $(it_1,it_2)$ of quaternions  with 
$t_1,t_2\in\rl^{^\times}$ and $|t_1|\ne |t_2|$. Under the action of the stabiliser 
$K_\zeta=\Sp_1{\times}\Sp_1{\times}\Sp_n$ the subspace 
$\gt v$ decomposes as $\gt v=\qv^n\oplus\qv^n$. 
Let us split  $v$ accordingly to this decomposition $v=v_1+v_2$. 
Here both $v_1$ and $v_2$ are rows.  
Then the restrictions of $F_1\,,F_3$, and $F_4$ %%%% ,F_5$ 
to  $\gt v\times\{\zeta\}$ are proportional to 
$$
%%%% \begin{array}{l}
t_1^2+t_2^2, \  %%%%% \quad\  \ \ t_1^4+t_2^4, \\
|v_1|^2+|v_2|^2, \  %%%% \qquad \quad
%%% |v_1|^4+v_1v_2^*v_2v_1^*+v_2v_1^*v_1v_2^*+|v_2|^4, \\
t_1^2|v_1|^2+t_2^2|v_2|^2. %%%% , \enskip 
%%t_1^2 |v_1|^4+t_2^2v_1v_2^*v_2v_1^*+t_1^2v_2v_1^*v_1v_2^*+t_2^2|v_2|^4.   \\
%%% \end{array}
$$  
Since  $|t_1|$ and $|t_2|$ are distinct, 
$t_1^2|v_1|^2+t_2^2|v_2|^2$ is not proportional
to $(t_1^2+t_2^2)(|v_1|^2+|v_2|^2)$. 
%%
% In order 
% to handle $F_5$, we restrict ourselves to 
% rank one $v$.
% Suppose $v=\left(\begin{array}{cccc}
 %                h_1 & 0 & \ldots & 0 \\
 %                h_2 & 0 & \ldots & 0 \\
 %                \end{array}\right)$, with $h_1,h_2\in\qv$. 
% Then 
% $$
% F_5(v,\zeta)=t_1^2|h_1\overline{ih_1}|^2+t_1^2|h_2\overline{ih_1}|^2+
% t_2^2|h_1\overline{ih_2}|^2+t_2^2|h_2\overline{ih_2}|^2.
% $$                 
% Note that $h_1^*=\overline{h_1}$ and 
% $h_1h_2^*h_2h_1^*=h_1|h_2|^2h_1^*=|h_1|^2|h_2|^2$.
% Hence $F_6(v,\zeta)=F_3^2(v,\zeta)$ and we only need to show that 
% $F_5$ does not lie in the subspace generated by $F_3F_4, F_1F_3^2$ 
% when restricted to $\qv v\times\rl\zeta$. This 
% subspaces is generated by 

\medskip

\noindent  
{\bf 11.} \  Here $K_\zeta=({\rm U}_1{\times}){\rm U}_1{\times}\SU_n$, 
$K_*(\gt n)=({\rm U}_1{\times})\SU_{n-2}$ for $n>2$ and 
$K_*(\gt n)$ is trivial for $n=2$.
Note that the (additional) torus lies in a
generic stabiliser, hence, the invariants 
do not differ. Thus we may (and will) assume that 
$K$ is not semisimple. (In case $n=2$ the torus 
is necessary.) 
 
The codimension of a generic orbit is $5$. On 
$\gt z$ there are two generating $K$-invariants: $F_1=i\tr(z)$ 
and $F_2=\tr(z^2)$. The subspace $\gt v$
is a real form of $\gt g_1$ arising from a symmetric decomposition 
$\gt g=\gt g_0\oplus\gt g_1$ with 
$(\gt g,\gt g_0)=(\gt{sl}_{n+2},\gt{sl}_n{\oplus}\gt{sl}_2)$.
Here there are two generators $F_3=|v|^2$ 
and $\det(v \bar v^t)$ or, equivalently, $F_4=\tr((vv^*)^2)$.  Thus
$\rl[\gt z]^K=\rl[F_1,F_2]$ and $\rl[\gt v]^K=\rl[F_3,F_4]$. 
In order to describe the ``mixed" generators, we explore the same method 
as in item~10.  

We have $\gt v(\cp)=\cp^2{\otimes}\cp^n\oplus(\cp^2{\otimes}\cp^n)^*$ 
and $\cp[\gt v(\cp)]^{\GL_n}=\cp[H_1,H_2,H_3,H_4]$ with all
four generators being of degree $2$.  On the subspace 
$\left<H_1,\ldots,H_4\right>_\cp$ 
the group  $\SL_2$ acts as on $\cp\oplus\gt{sl}_2$. 
The trivial subrepresentation corresponds to $F_3$.
On the other hand, $F_1$ comes from the trivial 
subrepresentation of $K$ in $\gt z$. Apart from these two
generators,
we have an action of $\SL_2$ on $\fs\fl_2\oplus\fs\fl_2$ or
of $\SO_3$ on $\bC^3\oplus\bC^3$. This gives rise to 
three invariants: one on the first copy of $\bC^3$, 
another on the second copy, 
and one ``mixed", $F_5'$, of bi-degree $(2,1)$ in $(\gt v,\gt z)$.
Thereby $\bR[\fn]^K$ is generated by 
$\bR[V]^K$, $\bR[\fz]^K$, and $F_5'$.
The subspace $(S^2(V(\bC))\otimes \fz(\bC))^{K(\bC)}$ 
is two-dimensional and has a basis $\{F_5',F_1F_3\}$. 
In order to complete the description, 
it is sufficient to show that $F_1F_3$ and $F_5$ are 
linear independent. Indeed, if $\tr(z)=0$, then 
$F_1F_3=0$, but $F_5$ is not.
 
\medskip

\noindent 
{\bf 12.} \ Here $K_\zeta={\rm U}_1{\times}{\rm U}_1{\times}\Sp_n$, 
$K_*(\gt n)=\Sp_{n-2}$, and the codimension of a generic orbit is $6$. 
This case is very similar to case~10. 
The complexification $\gt v(\cp)=(\cp^2{\oplus}(\cp^2)^*)\otimes\cp^{2n}$
is the same as for $\gt v$ in item~10 and the $K(\cp)$-representation 
on $\gt v(\cp)$ is obtained by the restriction of the defining representation 
of $\Sp_4(\cp)$ to a symmetric subgroup $\GL_2$.  
Again $\cp[\gt v(\cp)]^{\Sp_{2n}}=\cp[H_1,\ldots,H_6]$,
where all generators are of degree $2$, 
and $\GL_2$ acts on $\left<H_1,\ldots,H_6\right>_\cp$ as on
$\Lambda^2(\cp^2{\oplus}(\cp^2)^*)$, i.e., as on 
$\cp\oplus(\cp{\oplus}\cp^*)\oplus \gt{sl}_2$.
On  the first summand $\cp$ the action is trivial and it 
corresponds to $F_3=|v|^2$. On $\gt z$ there are two 
generators: $F_1=\tr(z), F_2=\tr(z^2)$.
Apart from this there are two invariants of degree 
$4$ on $\gt v$, one comes from $\cp\oplus\cp^*$, 
another from $\gt{sl}_2$; plus an invariant $F_6$
of bi-degree $(2,1)$ in $(\gt v,\gt z)$, which corresponds 
to the pairing between $\gt{sl}_2\subset \left<H_1,\ldots,H_6\right>_\cp$
and $\gt{sl}_2\subset\gt z(\cp)$. There are at most $6$ generators. Since 
we should have at least $6$ algebraically independent ones, they are 
algebraically independent.  As $F_6$ we take $(zv,v)$,
where $(\,\,,\,)$ is a $K$-invariant scalar product on $\gt v$.
Clearly $F_6$ is non-zero on $\gt v\times\gt{sl}_2$ and hence 
linear independent with $F_1F_3$. The polynomial can be 
expressed as $F_6=\tr(v^*izv)$. 
 
To complete the proof we need to show that 
$\rl[\gt v]^K=\rl[F_3,F_4,F_5]$. 
Recall that  $F_4=\tr((vv^*)^2)$. 
This is an invariant of $K$ and,
taking into account an isomorphism 
$\cp^2\otimes_{\cp}\qv^n\cong \qv^2\otimes_{\qv}\bH^n$, one
concludes that it is also $\Sp_2{\times}\Sp_n$-invariant.  
 
On the other hand, 
$\cp^2\otimes\qv^n\cong \cp\otimes_\rl \qv^n$ as a 
representation of ${\rm U}_1\times(\Sp_1{\times} \Sp_n)=K$.
One may say that $\gt v$ is a sum of 
$\qv^n=\rl^{4n}$ and $i\rl^{4n}$, where  
${\rm U}_1$ acts via multiplications by $\cos\gamma + i\sin\gamma$. 
Such situation appeared in case 8. This 
decomposition leads to the construction of $F_5$. 
Since the degrees of generating invariants have already been found out, 
it remains to show that $F_3,F_4,F_5$ are algebraically independent 
or  that $|v|^4,F_4,F_5$ are linearly independent.
If they were dependent, then the complex groups  
$\Sp_4\times \Sp_{2n}$ and $\bC^{^\times}\times \SO_{4n}$ 
would have had the same invariants on 
$\cp^4\otimes\cp^{2n}\cong (\cp{\oplus}\cp^*)\otimes\cp^{4n}$. 
But a subalgebra of $\gt{so}_{8n}(\bC)$, generated by the
Lie algebras of these  
two subgroups, coincides with $\gt{so}_{8n}(\cp)$ itself 
and neither $F_4$ nor $F_5$ is $\gt{so}_{8n}(\cp)$-invariant. 
Therefore $F_3$, $F_4$, and $F_5$ are algebraically independent 
and this completes the proof.    \hfill{$\Box$}

\begin{rmk} 
For the purposes of Section~4 we need a slightly modified
set of generators in case~(12). 
As was already mentioned, $\gt v$ can be viewed as 
$\rl^2\otimes_{\rl}\qv^n$ and $K$ as $\SO_2{\times}\Sp_1{\times}\Sp_n$.
Also $\gt z$ decomposes as $\gt{sp}_1\oplus\rl$. 
Now $F_1=t$, where $t$ is a coordinate on $\rl\subset\gt z$ and
the new $F_2=|z_0|^2$ with $z_0\in\gt{sp}_1$. The next generator, 
$F_3$, we leave as it is. 
We can present $v\in\gt v$ as a $2{\times}n$ quaternionic matrix, 
$$
v=\left(\begin{array}{l}
   v_1 \\
   v_2 \\
   \end{array}\right),
$$
where $v_1,v_2\in\qv^n$ are row-vectors. 
It can be assumed that $\trans v_1,\trans v_2$
are the same as $x,y$ in Theorem~\ref{generators}.  
In this terms 
$F_4=|v_1|^4 +2|v_1v_2^*|^2 + |v_2|^4$.
There is a unique (up to a scalar) $\Sp_1{\times}\Sp_n$-invariant 
pairing between $\qv^n$ and $\qv^n$. In terms of
real column-vectors $x,y$ it was expressed as $\trans x y$.
In terms of quaternionic row vectors $v_1,v_2$ it becomes 
exactly  $\RE(v_1v_2^*)$. Hence 
$F_5 =|v_1|^2|v_2|^2-\big(\RE(v_1v_2^*)\big)^2$. 
For the new generator $F_6$ we take a 
$K$-invariant polynomial $\RE((v_1v_2^*-v_2v_1^*)z_0)$
(here again $z_0\in\gt{sp}_1$).
As well as the old $F_6$ of Theorem~\ref{generators},
$\RE((v_1v_2^*-v_2v_1^*)z_0)$ is non-zero on $\gt v\times\gt{sp}_1$ and 
hence linearly independent with $F_1F_3$. 
\end{rmk}

\newpage

% \bigskip

% The invariants for $(\fn, K)$ are given by
% $\rho(v,z)=\left(\rho_\fv(v); \rho_{\fv,\fz}(v,z);|z|^2\right)$
% with
% $$
% \begin{array}{| c | c | c | l  r |}
% \hline
% &K&\fn=\fv\oplus\fz&\rho_\fv& \rho_{\fv,\fz} 
% \\
% \hline
%  4&\rm U_2&\bC^2\oplus \su_2&|v|^2&iv^*zv
% \\
% 5&\Sp_2&\bH^2\oplus H\!S^2_0\bH^2&|v|^2&v^*zv 
% \\
% 6&\rm U_2\times \SU_2&\bC^2\otimes\bC^2\oplus \su_2&\tr(v^*v)\,,\,\tr(v^*v)^2&i\tr(vzv^*)
% \\
% 7&\SU_2\times \SU_n &\bC^2\otimes\bC^n\oplus\su_2&\tr(v^*v)\,,\,\tr(v^*v)^2&i\tr(vzv^*)
% \\
% 8&\rm U_2\times\text{Sp}_n&\bC^2\otimes\bC^{2n}\oplus \su_2&\tr(v^*v)\,,\,\tr(v^*v)^2& i\tr(vzv^*)
% \\
% 9&\SO_2\times \text{Spin}_7&\bR^2\otimes\bR^8\oplus \bR^7&|v|^2\,,\,|v_1|^2|v_2|^2-\lan v_1,v_2\ran^2&\lan[v_1,v_2],z\ran
% \\
% \hline
% 10&\SO_3&\bR^3\oplus\Lambda^2\bR^3&|v|^2&\lan v,z\ran
% \\
% 11&\SU_3&\bC^3\oplus\Lambda^2\bC^3&|v|^2&\RE\lan z,v\ran\,,\,\IM\lan z,v\ran
% \\
% 12&\rm G_2&\IM\bO\oplus\IM\bO&|v|^2&\lan z,v\ran\\
% \hline
% \end{array}
% $$
% Remark: line 7 is for $n\ge 3$ only.

\end{document}